\documentclass[A4,12pt]{article}
\usepackage{amsfonts,bm}
\usepackage{amsmath,amsthm}
\usepackage{amssymb,amscd}
\usepackage{mathtools}
\usepackage{blkarray, bigstrut}
\usepackage{enumerate}
\usepackage{color}
\usepackage{cite}
\usepackage{esint}
\usepackage{tikz}
\usepackage{tikzscale}
\usepackage[absolute,overlay]{textpos}
\usepackage{multicol}
\usepackage{framed}
\usepackage{lscape}
\usepackage{caption}
\usepackage{subcaption}
\usepackage[inline]{enumitem}
\usepackage{float}
\usepackage{mathtools}
\usepackage[title]{appendix}
\usepackage{relsize}
\usepackage{etoolbox}
\usepackage{MnSymbol}
\usepackage{enumitem}
\usepackage{pgf}
\usepackage{array}
\usepackage[utf8]{inputenc}
\usetikzlibrary{calc,fit,matrix,arrows,automata,positioning}

\makeatletter
\newcommand{\rmnum}[1]{\romannumeral #1}
\newcommand{\Rmnum}[1]{\expandafter\@slowromancap\romannumeral #1@}

\def\Xint#1{\mathchoice
{\XXint\displaystyle\textstyle{#1}}%
{\XXint\textstyle\scriptstyle{#1}}%
{\XXint\scriptstyle\scriptscriptstyle{#1}}%
{\XXint\scriptscriptstyle\scriptscriptstyle{#1}}%
\!\int}
\def\XXint#1#2#3{{\setbox0=\hbox{$#1{#2#3}{\int}$}
\vcenter{\hbox{$#2#3$}}\kern-.5\wd0}}

\def\dashint{\Xint-}

\theoremstyle{definition}
\newtheorem{theorem}{Theorem}[section]
\newtheorem{corollary}[theorem]{Corollary}
\newtheorem{lemma}[theorem]{Lemma}

\newtheorem{proposition}[theorem]{Proposition}

\theoremstyle{remark}
\newtheorem{remark}{Remark}[section]

\newcommand{\qbinom}[2]{\left[{{#1}\atop#2}\right]}

\tikzset{%
  highlight/.style={rectangle,rounded corners,fill=red!15,draw,fill opacity=0.5,thick,inner sep=0pt}
}

\begin{document}
\renewcommand{\arraystretch}{1.5}
\title{Distribution of a second-class particle's position in the two-species ASEP with a special initial configuration}
\author{\textbf{Eunghyun Lee\footnote{eunghyun.lee@nu.edu.kz}}\\[3pt]and \\[3pt] \textbf{Zhanibek Tokebayev} \\[8pt] {\text{Department of Mathematics, Nazarbayev University}}
                                         \date{}   \\ {\text{Kazakhstan}}   }

\date{}
\maketitle
\begin{abstract}
\noindent In this paper, we consider the two-species asymmetric simple exclusion process (ASEP) consisting of $N-1$ first-class particles and one second-class particle. We assume that all particles are located at arbitrary positions but the second-class particle is the rightmost particle at time $t=0$. We find the exact formula of the distribution of the second-class particle's position at time $t$ by directly using the transition probabilities of the two-species ASEP, which is a different approach from the coupling method Tracy and Widom used in \cite{Tracy-Widom-2009}.
\end{abstract}
\section{Introduction}
The multi-species asymmetric simple exclusion process is a generalization of the asymmetric simple exclusion process (ASEP)  in the sense that each particle may belong to a different \textit{species}. In principle, particles follow the rules of the ASEP, that is, a particle tries to jump to the right with probability $p$ or to the left with probability $q=1-p$ after exponential waiting time with parameter $1$. But a particle belonging to a \textit{higher} species is given priority when it tries to jump to a site already occupied by a \textit{lower} species. To be more specific, let us denote  a species by an integer $l$. If a particle belonging to a species $l$ tries to jump to a site already occupied by a particle belonging to $l'<l$, then the  particle belonging to $l$ can jump by interchanging its position with the position of the particle belonging to $l'$. However, if it tries to jump to a site occupied by a particle $l''\geq l$, the jump is not allowed. In the two-species ASEP, it is common that a \textit{lower} species particle is called a second-class particle and a \textit{higher} species particle is called  a first-class particle.

 The multi-species ASEP is mainly being considered on the one-dimensional finite integer lattice or infinite lattice $\mathbb{Z}$. Some featured results about the multi-species ASEP on a finite integer lattice are found in  \cite{Ayyer-Finn-Roy-2018,Derrida-Evans-Hakim-Pasquier-1993,Evans-Ferrari-Mallick-2009,Ferrari-Martin,Prolhac-Evans-Mallick-2009} where the steady state of the process is one of the main interests. In the multi-species ASEP on $\mathbb{Z}$, the exact formulas of the transition probability and some probability distributions  were studied in  \cite{Chatterjee-Schutz-2010,Gier-Mead-Wheeler-2021,Kuan-2020,Lee-2017,Lee-2018,Lee-2020,Lee-Raimbekov-2022,Tracy-Widom-2009, Tracy-Widom-2013}. Also, the two-species with a special initial configuration, called the step initial condition, has drawn an attention in studying a partial differential equation as the hydrodynamic limit of the one-dimensional ASEP \cite{Ferrari-Kipnis-1995,Mountford-Giuol-2005}, the competition interface in the last-passage percolation \cite{Ferrari-Pimentel-2005} and some asymptotics in the realm of the KPZ statistics \cite{Nejjar-2020}. Moreover, the multi-species ASEP was recently studied in more general context, the coloured stochastic vertex model \cite{Borodin-Wheeler,Borodin-Bufetov}.

In this paper, we consider the two-species ASEP.   The main purpose of this paper is to provide the exact formula of the probability distribution of the second-class particle's position at time $t$ for an $N$-particle system with $(N-1)$ first-class particles and one second-class particle. The initial configuration of our interest is that the second-class particle is the rightmost particle and the positions of the particles are arbitrarily fixed. In fact, Tracy and Widom have already found the probability distribution of the second-class particle's position at time $t$ for an infinite system with all positive sites occupied by first-class particles and the origin occupied by a  second-class particle and all negative sites unoccupied at $t=0$ \cite{Tracy-Widom-2009}. The method Tracy and Widom used is the coupling method for two (single-species) ASEPs with a single discrepancy in their initial configurations where the discrepancy can be  viewed as the second-class particle. But the coupling method in \cite[Section 2]{Tracy-Widom-2009} is for a system with only one second-class particle, and it is not clear if the coupling method in Tracy and Widom's work can be extended to a system with multiple second-class particles. In this paper, we provide a different method for the  probability distribution of the second-class particle's position. Our approach  is to use the transition probabilities of the two-species ASEP. Since the transition probabilities of the two-species ASEP are fully known \cite{Lee-2020}, it should be possible to find the probability distribution of a second-class particle's position even when there are multiple second-class particles initially. The method using the transition probabilities of the two-species ASEP is straightforward in the sense that the required probability is obtained by summing the transition probabilities over all possible configurations but a \textit{huge} computation is needed for a large system or the proof for a general $N$-particle system.  One contribution of this paper is to provide a method for the \textit{huge} computation although we focus on the system with one second-class particle. It remains to be seen how much useful this method is for multiple second-class particle cases, but in the authors' another ongoing work, we could successfully obtain the exact formula of the probability distribution of the leftmost second-class particle's position in a small system (with one first-class particle and three second-class particles) by using some techniques in this paper. Also, since we assume that  the positions of the particles are arbitrarily fixed initially (but the second-class particle is the rightmost particle) in this paper, we expect that our formula can be used for a different special initial condition such as \textit{flat initial condition}.

To overview our approach and state the main result, let us denote a state of the process by $(X,\nu)$ where $X = (x_1,\dots, x_N)$ with $x_1<\cdots <x_N$  represents the positions of $N$ particles, and $\nu$ is a permutation  of  the multi-set $[1,2,\dots,2]$ with cardinality $N$, representing the order of particles. (A permutation of a multi-set is an ordered arrangement of all elements of the multi-set. Thus, the multi-set $[1,2,\dots,2]$ with cardinality $N$ has $N$ permutations.) Here, ``1" represents a particle belonging to species 1 (a second-class particle) and ``2" represents a particle belonging to species 2 (a first-class particle). Let us denote by $\nu(n)$ the $n$th number in the permutation $\nu$ and denote by  $\nu_n$  the permutation $\nu$ with  $\nu(n) = 1$. The  probability that the process is at state $(X,\pi)$ at time $t$ given that the initial configuration is $(Y,\nu)$ is denoted by  $P_{(Y,\nu)}(X,\pi;t)$.  For simplicity of notations,  we write   $P_{(Y,\nu_N)}(X,\nu_n;t) = P_Y(X,\nu_n;t)$ because, in this paper, we will consider only $\nu_N$ for the initial order of particles. One of the authors of this paper found the transition probabilities of the $N$-particle multi-species ASEP with arbitrary combination of species \cite{Lee-2020}. In particular, the transition probabilities of interest in this paper are given by
\begin{equation}\label{533pm813}
P_Y(X,\nu_n;t)= \sum_{\sigma \in \mathcal{S}_N}\dashint_{c}\cdots\dashint_{c} [\mathbf{A}_{\sigma}]_{\nu_n,\nu_N}\prod_{i=1}^N\Big(\xi_{\sigma(i)}^{x_i - y_{\sigma(i)}-1} e^{(\frac{p}{\xi_i}+q\xi_i - 1)t}\Big)d\xi_1\cdots d\xi_N
\end{equation}
where $c$ is a counterclockwise circle centered at the origin with sufficiently small radius $r$ and $[\mathbf{A}_{\sigma}]_{\nu_n,\nu_N}$ is in the form of
\begin{equation*}
\prod_{(\beta,\alpha)}R_{\beta\alpha}
\end{equation*}
where the product is taken over all inversions $(\beta,\alpha)$ in $\sigma$ and $R_{\beta\alpha}$ is one of
\begin{equation}\label{1100am911}
S_{\beta\alpha} =  -\frac{p+q\xi_{\alpha}\xi_{\beta} - \xi_{\beta}}{p+q\xi_{\alpha}\xi_{\beta} - \xi_{\alpha}},~~ pT_{\beta\alpha} = \frac{p(\xi_{\beta}-\xi_{\alpha})}{p+q\xi_{\alpha}\xi_{\beta} - \xi_{\alpha}},~~ Q_{\beta\alpha} =\frac{(p-q\xi_{\beta})(\xi_{\alpha}-1)}{p+q\xi_{\alpha}\xi_{\beta} - \xi_{\alpha}}
\end{equation}
or zero \cite{Lee-2020,Lee-Raimbekov-2022}. How to choose $R_{\beta\alpha}$ from $S_{\beta\alpha}, pT_{\beta\alpha}, Q_{\beta\alpha}, 0$ for each inversion $(\beta,\alpha)$ depends on $\sigma$ and $\nu_n$. See Appendix \ref{351am129} for the details on $[\mathbf{A}_{\sigma}]_{\nu_n,\nu_N}$ or directly see Theorem 1.2 (a), Theorem 1.4, Theorem 1.6, Theorem 1.7, and Proposition 1.8 in \cite{Lee-Raimbekov-2022} for the explicit formulas of $[\mathbf{A}_{\sigma}]_{\nu_n,\nu_N}$. Here, $\sum_{\sigma \in \mathcal{S}_N}$ implies the sum over all permutations of the symmetric group $\mathcal{S}_N$ and  $\dashint$ implies $(1/2\pi i)\int$ throughout this paper. We recall that, in the case of the single-species ASEP,  the formula of the transition probabilities,
\begin{equation*}
P_Y(X;t)= \sum_{\sigma \in \mathcal{S}_N}\dashint_{c}\cdots\dashint_{c} {A}_{\sigma}\prod_{i=1}^N\Big(\xi_{\sigma(i)}^{x_i - y_{\sigma(i)}-1} e^{(\frac{p}{\xi_i}+q\xi_i - 1)t}\Big)d\xi_1\cdots d\xi_N
\end{equation*}
where
\begin{equation}\label{126am830}
A_{\sigma} = \prod_{(\beta,\alpha)}S_{\beta\alpha}
\end{equation}
was obtained by Tracy and Widom \cite{Tracy-Widom-2008}.
\subsection{Main results}
Let us denote by $\mathbb{P}$ the probability measure of the process with the initial state $(Y, \nu_N)$ and let $\eta(t)$ be the random variable of the  position of the second-class particle at time $t$. Then, it is obvious that
 \begin{equation}\label{1051pm527}
 \mathbb{P}(\eta(t) = x) = \sum_{\substack{X\,\textrm{with} \\x_N = x}} P_{Y}(X,\nu_N;t) + \cdots + \sum_{\substack{X\,\textrm{with} \\x_2 = x}} P_{Y}(X,\nu_{2};t)  + \sum_{\substack{X\,\textrm{with} \\x_1 = x}} P_{Y}(X,\nu_1;t).
\end{equation}
It is not hard to compute (\ref{1051pm527}) for a small system such as $N=2$, but even for $N=3,4$, it requires a tedious and long computation together with contour deformations. Moreover, it is not easy to even conjecture the formula for general $N$ from the results for $N=2,3,4$. One contribution of this paper is to provide a novel idea to evaluate the sum (\ref{1051pm527}) to enable us to prove the formula for general $N$. We introduce some notations that will be used throughout the paper.
\subsubsection{Notations}
We write $\mathcal{S}_U$ for the set of all permutations on a finite set $U$. If $U = \{1,\dots, n\}$, we simply write $\mathcal{S}_U =\mathcal{S}_n$.
For a set $U = \{u_1,\dots, u_n\}$ of positive integers with $u_1<\cdots <u_n$, we write $\bm{\xi}_U = (\xi_{u_1},\dots, \xi_{u_n})\in \mathbb{C}^n$ and $Y_U = (y_{u_1},\dots, y_{u_n})$, an ordered $n$-tuple of integers with $y_{u_1}<\cdots < y_{u_n}$. If $U = \{1,\dots, N\}$ where $N$ is the total number of particles in the system, then we simply write $\bm{\xi}_U = \bm{\xi}$ and $Y_U = Y$.
For a nonempty finite set of positive integers $U$,  define $g_U: U \to \mathbb{Z}$ by $g_U(u) = i$ where $i$ implies that $u$ is the $i$th smallest number in $U$.  For a nonempty subset $S$ of $U$,  we define
\begin{equation*}
\Sigma_U(S) := \sum_{s \in S}g_U(s)
\end{equation*}
and $\Sigma_U(\emptyset) := 0$. For example, if $U=\{2,4,7\}$, then $g_U(2)=1, g_U(4) = 2, g_U(7) = 3$, and for $S=\{2,7\} \subset U$,
\begin{equation*}
\Sigma_{U}(S) := g_U(2) +g_U(7) = 4.
\end{equation*}
If $U =\{1,\dots, n\}$ for a given positive integer $n$, then  we write $\Sigma_U(S) = \Sigma(S)$ for simplicity, which implies the sum of all elements in $S$.
For $U = \{u_1,\dots, u_n\}$, let
\begin{equation*}
J(\bm{\xi}_U)=J(\xi_{u_1},\dots, \xi_{u_n}) =\frac{1}{\displaystyle \prod_{u \in U}(\xi_{u}-1)},~I(\bm{\xi}_U)  = (\xi_{u_1}\cdots \xi_{u_n} - 1)J(\bm{\xi}_U)
\end{equation*}
and
\begin{equation*}
W_{t,x,Y_U}(\bm{\xi}_U) = \prod_{i=1}^n\Big(\xi_{u_i}^{x - y_{u_i}-1}e^{(\frac{p}{\xi_{u_i}}+ q\xi_{u_i} -1)t}\Big).
\end{equation*}
We define
\begin{equation*}
\prod_{i = m}^n a_i := \begin{cases}
a_ma_{m+1}\cdots a_n &~\textrm{if}~n \geq m,\\[10pt]
1 &~\textrm{if}~n = m-1.
\end{cases}
\end{equation*}
\subsubsection{Statements of results}
\begin{theorem}\label{102am830}
If the initial configuration is $(Y, 2\cdots21)$, then the probability distribution of the second-class particle's position at time $t$, denoted by $\eta(t)$,  is
\begin{equation}\label{510pm829}
\mathbb{P}(\eta(t) = x) = \sum_{\emptyset \neq  S\subset \{1,\dots, N\}} c_S~\dashint_c \cdots \dashint_c \Big(\prod_{\substack{\alpha<\beta,\\ \alpha,\beta \in S}} T_{\beta\alpha}\Big)I(\bm{\xi}_S)W_{t,x,Y_S}(\bm{\xi}_S) d\bm{\xi}_S
\end{equation}
 where
 \begin{equation}\label{511pm829}
c_S=
\begin{cases}
~\displaystyle \bigg(\prod_{i=1}^{|S|-1}(q^i-p^i)\bigg)\Big(\frac{q}{p}\Big)^{\Sigma(S^c) - |S^c|(|S^c|+1)/2}&~~\textrm{if}~N \in S,\\[20pt]
~\displaystyle \frac{1}{p^{|S|}}\bigg(\prod_{i=1}^{|S|}(q^i-p^i)\bigg)\Big(\frac{q}{p}\Big)^{\Sigma(S^c) - |S^c|(|S^c|-1)/2-N}&~~\textrm{if}~N \notin S
\end{cases}
\end{equation}
where $S^c = \{1,\dots, N\} \setminus S$ and $T_{\beta\alpha}$ is given in (\ref{1100am911}).
\end{theorem}
The formula (11) in \cite{Tracy-Widom-2009} by Tracy and Widom for an infinite system with step initial condition was given by the contour integrals with \textit{large} contours. On the other hand, our formula (\ref{510pm829}) which is for a finite system is given by the contour integrals with \textit{small} contours. We did not try to rediscover Tracy and Widom's formula from our result but we believe that it is possible.  Also, to the best of authors' knowledge,  the probability distribution of the second-class particle for \textit{flat initial condition} has not been obtained. It would be interesting to see if our formula can be used to find a formula for \textit{flat initial condition} (see \cite{Lee-2010} for some exact formulas in the single-species ASEP with \textit{flat initial condition}).

A simple but interesting  result can be obtained in the limit $p \to \frac{1}{2}$.
\begin{corollary}\label{525pm829}
In the two-species symmetric exclusion process
\begin{equation}\label{522pm829}
\mathbb{P}(\eta(t) = x) = \dashint_c \xi^{x - y_{N}-1}e^{(\frac{1}{2\xi}+ \frac{\xi}{2} -1)t}d\xi.
\end{equation}
\end{corollary}
Note that $y_N$  in (\ref{522pm829}) is the initial position of the second-class particle. It is interesting that the formula (\ref{522pm829}) is the same as   the probability distribution of the continuous-time symmetric random walk on $\mathbb{Z}$ when the random walk's initial position is $y_N$. Hence, roughly speaking, our result implies that the second-class particle behaves like the symmetric random walk although it is affected by the first-class particles.

\subsection{Organization of the paper}
In Section \ref{642pm826}, we provide some lemmas for the proof of Theorem \ref{102am830}. In Section \ref{1249am830}, we find the exact formula of the probability distribution for a small system with $N=3$. Section \ref{1249am830} will serve as a warmup for proving the general formula for $N$-particle system. The main idea of the proof for the general formula in Section  \ref{101am830} is based on the techniques in Section \ref{1249am830}. Finally,  we prove our main result, Theorem \ref{102am830}, in Section \ref{101am830}.
\section{Lemmas}\label{642pm826}
In this section, we provide some results which will be used for the proof of our main theorem. First, Lemma \ref{319pm1312023} is a part of the proof of Lemma 3.1 in \cite{Tracy-Widom-2008} and is used to prove Lemma \ref{542pm828}.
\begin{lemma}\cite{Tracy-Widom-2008}\label{319pm1312023}
Suppose that $f(\xi_1,\dots,\xi_n)$ is analytic for all $\xi_i \neq 0$ and that for $i>k$,
\begin{equation*}
f(\xi_1,\dots,\xi_n)\Big|_{\xi_i \to (\xi_k-p)/q\xi_k} = O(\xi_k),
\end{equation*}
as $\xi_k \to 0$, uniformly when all $\xi_j$ with $j \neq k$ are bounded and bounded away from zero. Then,
\begin{equation*}
\dashint_{c}\cdots \dashint_{c}\bigg(\prod_{i<j}\frac{\xi_j - \xi_i}{p+q\xi_i\xi_j - \xi_i}\frac{f(\xi_1,\dots,\xi_n)}{\prod_{i}(1-\xi_i)}\bigg)\bigg|_{\xi_n \to (\xi_k-p)/q\xi_k}d\xi_1\cdots d\xi_k \cdots d\xi_{n-1} =0.
\end{equation*}
\end{lemma}
If, in particular,
\begin{equation*}
f(\xi_1,\dots,\xi_n) = \frac{1}{(\xi_1,\dots,\xi_n)^z} \times W_{t,x,Y_{\{1,\dots,n\}}}(\xi_1,\dots,\xi_n),~z = 1,2,\dots,
\end{equation*}
then $f(\xi_1,\dots,\xi_n)$ satisfies the hypothesis of Lemma \ref{319pm1312023}. (See Proof of Theorem 3.2 in \cite{Tracy-Widom-2008}.) \\ \\
Now, we provide Lemma \ref{542pm828} which will be used to prove Lemma \ref{334am513}.
\begin{lemma}\label{542pm828}
Let $U = \{u_1,\dots, u_n\}\subset \{1,2,\dots \}$ for $n\geq 2$, $T_{\beta\alpha}$ be given in (\ref{1100am911}),  $A_\sigma$ with  a permutation $\sigma$ on $U$ be given by  (\ref{126am830}), and
\begin{equation}\label{105am528}
c_S = q^{n(n-1)/2}\frac{q^{\Sigma_U(S^c)-n|S^c|}}{p^{\Sigma_U(S^c) - |S^c|(|S^c|+1)/2}}.
\end{equation}
If
\begin{equation}\label{354am814}
\begin{aligned}
&~~\sum_{z_1,\dots, z_{n-1}=1}^{\infty}\dashint_{c}\cdots \dashint_{c}\sum_{\sigma\in \mathcal{S}_U}A_{\sigma}\prod_{i=1}^{n-1}\big(\xi_{\sigma(1)}\cdots \xi_{\sigma(i)}\big)^{-z_i}W_{t,x,Y_U}(\bm{\xi}_U)d\bm{\xi}_U \\
&\hspace{1cm}=~\sum_{S \subset U}~c_S\dashint_{c}\cdots \dashint_{c}\bigg(\prod_{\substack{\alpha<\beta, \\\alpha,\beta \in S}}T_{\beta\alpha}\bigg)I(\bm{\xi}_S)W_{t,x,Y_S}(\bm{\xi}_S) d\bm{\xi}_S
\end{aligned}
\end{equation}
is true for $n$, then
\begin{equation}\label{355am814}
\begin{aligned}
&~~\sum_{z_1,\dots, z_{n}=1}^{\infty}\dashint_{c}\cdots \dashint_{c}\sum_{\sigma \in \mathcal{S}_U}A_{\sigma}\prod_{i=1}^{n}\big(\xi_{\sigma(1)}\cdots \xi_{\sigma(i)}\big)^{-z_i}W_{t,x,Y_U}(\bm{\xi}_U)d\bm{\xi}_U \\
&\hspace{1cm}=~1 + \sum_{S \subset U}~c_S\dashint_{c}\cdots \dashint_{c}\bigg(\prod_{\substack{\alpha<\beta, \\\alpha,\beta \in S}}T_{\beta\alpha}\bigg)J(\bm{\xi}_S)W_{t,x,Y_S}(\bm{\xi}_S) d\bm{\xi}_S
\end{aligned}
\end{equation}
is true for $n$. ($\sum_{S\subset U}$ implies the sum over all nonempty subsets $S \subset U$.)
\begin{proof}
 We will prove for the case of $U = \{1,\dots, n\}$ without loss of generality.
Let us write the left-hand side of (\ref{355am814}) as
\begin{equation}\label{7292pm826}
\sum_{z_n = 1}^{\infty}\bigg(\sum_{z_1,\dots, z_{n-1}=1}^{\infty}\dashint_{c}\cdots \dashint_{c}\sum_{\sigma \in \mathcal{S}_n}A_{\sigma}\prod_{i=1}^{n-1}\big(\xi_{\sigma(1)}\cdots \xi_{\sigma(i)}\big)^{-z_i}W_{t,x-z_n,Y_U}(\bm{\xi}_U)d\bm{\xi}_U\bigg).
\end{equation}
Since  we assume that (\ref{354am814}) is true, (\ref{7292pm826}) becomes
\begin{equation}\label{806pm826}
\sum_{z_n = 1}^{\infty}\sum_{\substack{S\subset \{1,\dots,n\},\\ S\neq \emptyset}}~c_S\dashint_{c}\cdots \dashint_{c}\bigg(\prod_{\substack{\alpha<\beta, \\\alpha,\beta \in S}}T_{\beta\alpha}\bigg)I(\bm{\xi}_S)W_{t,x-z_n,Y_S}(\bm{\xi}_S) d\bm{\xi}_S.
\end{equation}
Note that any nonempty subset of $\{1,\dots, n\}$ is a nonempty subset ${S}'$ of $\{1,\dots, n-1\}$ or ${S}'\cup \{n\}$ or $\{n\}$. Let us consider a pair of integral terms for a set $\emptyset \neq S'=\{s_1,\dots, s_r\}\subset \{1,\dots, n-1\}$ and   $S=S'\cup \{n\}$ in (\ref{806pm826}):
\begin{eqnarray}
& &\sum_{z_n = 1}^{\infty}c_S\dashint_{c}\cdots \dashint_{c}\bigg(\prod_{\substack{\alpha<\beta, \\\alpha,\beta \in S}}T_{\beta\alpha}\bigg)I(\bm{\xi}_S)W_{t,x,Y_S}(\bm{\xi}_S) \big(\xi_{s_1}\cdots \xi_{s_r}\xi_n\big)^{-z_n} d\bm{\xi}_S, \label{100am827}\\
& & \sum_{z_n = 1}^{\infty}c_{S'}\dashint_{c}\cdots \dashint_{c}\bigg(\prod_{\substack{\alpha<\beta, \\\alpha,\beta \in S'}}T_{\beta\alpha}\bigg)I(\bm{\xi}_{S'})W_{t,x,Y_{S'}}(\bm{\xi}_{S'})\big(\xi_{s_1}\cdots \xi_{s_r}\big)^{-z_n} d\bm{\xi}_{S'}. \label{100am8279}
\end{eqnarray}
In (\ref{100am827}), if we enlarge the contour $c$ for the variable $\xi_n$ to a sufficiently large circle $C$ for the convergence of the sum over $z_n$, then we encounter a pole at $\xi_n=1$ that comes from $I(\bm{\xi}_{S})$ and poles $\xi_n = (\xi_{\alpha} - p)/q\xi_{\alpha}$ that come from $T_{n\alpha}$. By Lemma \ref{319pm1312023}, the residue at $\xi_n = (\xi_{\alpha} - p)/q\xi_{\alpha}$ is zero when integrated over $\xi_{\alpha}$.  Hence, (\ref{100am827}) equals
\begin{equation}\label{122am827}
\begin{aligned}
&\sum_{z_n = 1}^{\infty}c_S\dashint_{c}\cdots \dashint_{c}\dashint_{C}\bigg(\prod_{\substack{\alpha<\beta, \\\alpha,\beta \in S}}T_{\beta\alpha}\bigg)I(\bm{\xi}_S)W_{t,x,Y_S}(\bm{\xi}_S)\big(\xi_{s_1}\cdots \xi_{s_r}\xi_n\big)^{-z_n} d\bm{\xi}_S\\
& \hspace{1cm}-\sum_{z_n = 1}^{\infty}c_{S'}\dashint_{c}\cdots \dashint_{c}\bigg(\prod_{\substack{\alpha<\beta, \\\alpha,\beta \in S'}}T_{\beta\alpha}\bigg)I(\bm{\xi}_{S'})W_{t,x,Y_{S'}}(\bm{\xi}_{S'})\big(\xi_{s_1}\cdots \xi_{s_r}\big)^{-z_n} d\bm{\xi}_{S'}
\end{aligned}
\end{equation}
where $c_{S'}$ is due to the facts that
\begin{equation}\label{223am96}
c_S =p^{|S'|}\times c_{S'},\,\frac{1}{p^{|S'|}} = \Big(\prod_{\alpha \in S'}T_{n\alpha}\Big)\Big|_{\xi_n=1}.
\end{equation}
We note that the second term in (\ref{122am827}) can cancel with (\ref{100am8279}). Evaluating the first sum in (\ref{122am827}) and then shrinking $C$ back to $c$, we obtain the sum of  (\ref{100am827}) and (\ref{100am8279}) equal to
\begin{equation*}\label{125am82712}
 c_S\dashint_c\cdots \dashint_{c}\bigg(\prod_{\substack{\alpha<\beta, \\\alpha,\beta \in S}}T_{\beta\alpha}\bigg)J(\bm{\xi}_S)W_{t,x,Y_S}(\bm{\xi}_S) d\bm{\xi}_{S} \,+\,c_{S'}\dashint_c\cdots \dashint_{c}\bigg(\prod_{\substack{\alpha<\beta, \\\alpha,\beta \in S'}}T_{\beta\alpha}\bigg)J(\bm{\xi}_{S'})W_{t,x,Y_{S'}}(\bm{\xi}_{S'}) d\bm{\xi}_{S'}.
\end{equation*}
Finally, the sum of the integral for $S = \{n\}$ over $z_n$ in (\ref{806pm826}) is
\begin{equation*}
 \sum_{z_n=1}^{\infty} c_{\{n\}} \dashint_c \xi_{n}^{-z_n}W_{t,x,y_n}(\xi_n)d\xi_n~ =~1 +  \dashint_cJ(\xi_n)W_{t,x,y_n}(\xi_n)d\xi_n
\end{equation*}
because $W_{t,x,y_n}(1)=1$ and $c_{\{n\}}=1$. Hence, (\ref{806pm826}) equals the right-hand side of (\ref{355am814}).
\end{proof}
\end{lemma}
\begin{remark}\label{704pm913}
Recall that $W_{t,x,Y_U}(\bm{\xi}_U)$ is analytic for each variable $\xi_{u_i}$ except at the origin. If  $W_{t,x,Y_U}(\bm{\xi}_U)$ and $W_{t,x,Y_S}(\bm{\xi}_S)$ are replaced with $f(\bm{\xi}_U)$ and $f(\bm{\xi}_U)|_{\xi_i=1,\,i\notin S}$, respectively, for any analytic function $f(\bm{\xi}_U)$  for each variable $\xi_{u_i}$ inside a sufficiently large circle centered at the origin except at the origin, then  the constant 1 on the right-hand side of (\ref{355am814}) is replaced with $f(1,\dots, 1)$. The proofs are almost the same.
\end{remark}
For nonnegative integer $n$, we define
\begin{eqnarray*}
[n] =  \frac{p^n - q^n}{p-q}
\end{eqnarray*}
and
\begin{equation*}
[n]! =
\begin{cases}
 [n][n-1]\cdots [1]&~~~\textrm{if}~n\geq 1, \\[10pt]
1&~~~\textrm{if}~n=0,
\end{cases}
\end{equation*}
\begin{equation}\label{510pm822}
\qbinom{n}{m} =
\begin{cases}
\displaystyle \frac{[n]!}{[n-m]![m]!}&~~\textrm{if}~m\leq n, \\[20pt]
0&~~\textrm{if}~m> n.
\end{cases}
\end{equation}
Let $U = \{u_1,\dots, u_n\}\subset \{1,2,\dots \}$ and $S$ be a nonempty subset of $U$.
\begin{lemma}\label{334am513}
For  $U = \{u_1,\dots, u_n\}\subset \{1,2,\dots \}$, (\ref{354am814}) is true for all $n\geq 2.$
\begin{proof} We will prove by mathematical induction for the case of $U = \{1,\dots, n\}$ without loss of generality.
When $n=2$, a simple manipulation of the contour (that is, enlarging and then shrinking back after summation) shows that
\begin{equation}\label{402am829}
\begin{aligned}
& \sum_{z=1}^{\infty}\dashint_c \dashint_c\xi_1^{-z} W_{t,x,Y_{\{1,2\}}}(\xi_1,\xi_2)d\xi_1d\xi_2 \\
&\hspace{1cm}~=  \dashint_c \dashint_c \frac{1}{\xi_1-1}W_{t,x,Y_{\{1,2\}}}(\xi_1,\xi_2)d\xi_1d\xi_2 + \dashint_c W_{t,x,y_2}(\xi_2)d\xi_2
 \end{aligned}
\end{equation}
and
\begin{equation}\label{403am829}
\begin{aligned}
 & \sum_{z=1}^{\infty}\dashint_c \dashint_c S_{21}\xi_2^{-z}W_{t,x,Y_{\{1,2\}}}(\xi_1,\xi_2)d\xi_1d\xi_2 \\
  &\hspace{1cm}~=  \dashint_c \dashint_c \frac{1}{\xi_2-1}S_{21}W_{t,x,Y_{\{1,2\}}}(\xi_1,\xi_2)d\xi_1d\xi_2 + \frac{q}{p}\,\dashint_c W_{t,x,y_1}(\xi_1)d\xi_1,
 \end{aligned}
\end{equation}
where the single variable integrals in (\ref{402am829}) and (\ref{403am829}) are the residues at $\xi_1=1$ and $\xi_2=1$, respectively,
and the sum of (\ref{402am829}) and (\ref{403am829}) shows that (\ref{354am814}) holds for $n=2$.\\
\\
Suppose that (\ref{354am814}) holds for $n-1$. Let us decompose $\sum_{\sigma \in \mathcal{S}_n}$ in the left-hand side of (\ref{354am814}) as follows:
\begin{equation*}
\sum_{\sigma \in \mathcal{S}_n} (~~\cdots~~)= \sum_{l=1}^n\sum_{\substack{\sigma \in \mathcal{S}_n \\\textrm{with}\, \sigma(n) = l}}(~~\cdots~~).
\end{equation*}
For $l \in \{1,\dots, n\}$ and $\sigma \in \mathcal{S}_n$  with $\sigma(n) = l$, let $\sigma' = \sigma(1)\cdots\sigma(n-1) \in \mathcal{S}_{\{1,\dots,n\}\setminus \{l\}}$.
 Then,
\begin{equation*}
A_{\sigma} = A_{\sigma'}\times \prod_{l<\beta} S_{\beta l}
\end{equation*}
and the left-hand side of (\ref{354am814}) becomes
\begin{equation}\label{1128pm828}
\begin{aligned}
&\sum_{l=1}^n\dashint_c\bigg(\sum_{z_1,\dots, z_{n-1}=1}^{\infty}\dashint_{c}\cdots \dashint_{c}\sum_{\sigma' \in \mathcal{S}_{\{1,\dots,n\}\setminus \{l\}}}A_{\sigma' }\prod_{i=1}^{n-1}\big(\xi_{\sigma(1)}\cdots \xi_{\sigma(i)}\big)^{-z_i}\\
&\hspace{3cm}\times \underbrace{\Big(\prod_{l<\beta} S_{\beta l}\Big)W_{t,x,Y}(\bm{\xi})}_{(\star)}d\xi_{\sigma(1)}\cdots d\xi_{\sigma(n-1)}\bigg)d\xi_{l}.
\end{aligned}
\end{equation}
Note that $(\star)$ is analytic for each variable $\xi_{i}\neq \xi_l$ inside a sufficiently large circle centered at the origin except at the origin.  Hence, by \textit{Remark} \ref{704pm913} which supplements Lemma \ref{542pm828}, (\ref{1128pm828}) becomes
\begin{equation}\label{141am829}
\begin{aligned}
&\sum_{l=1}^n\dashint_c\Bigg(\bigg(\Big(\prod_{l<\beta} S_{\beta l}\Big)W_{t,x,Y}(\bm{\xi})\bigg)\bigg|_{\xi_i = 1,\,i\neq l}  + \sum_{\substack{S' \subset \{1,\dots,n\} \setminus \{l\},\\ S' \neq \emptyset}}c_{S'}\dashint_{c}\cdots \dashint_{c}\Big(\prod_{\substack{\alpha<\beta, \\\alpha,\beta \in S'}}T_{\beta\alpha}\Big)\\
&\hspace{2cm} \times ~J(\bm{\xi}_{S'}) \bigg(\Big(\prod_{l<\beta} S_{\beta l}\Big)W_{t,x,Y}(\bm{\xi})\bigg)\bigg|_{\xi_{i} = 1,\,i \notin S'\cup\{l\}} d\bm{\xi}_{S'}\Bigg)d\xi_l
\end{aligned}
\end{equation}
where
\begin{equation*}
c_{S'} =  q^{(n-1)(n-2)/2}\times \frac{q^{\sum_A(A\setminus S') - (n-1)|A\setminus S'|}}{p^{{\sum_A(A\setminus S') - |A\setminus S'|(|A\setminus S'| +1)/2}}}
\end{equation*}
with $A =\{1,\dots, n\}\setminus \{l\}$.
If  $<l>$ is  the number of elements in $A\setminus S'$ larger than $l$,
\begin{equation*}
\bigg(\Big(\prod_{l<\beta} S_{\beta l}\Big)W_{t,x,Y}(\bm{\xi})\bigg)\bigg|_{\xi_{i} = 1~\textrm{if}~i \notin S'\cup\{l\}} = \Big(\frac{q}{p}\Big)^{<l>}\Big(\prod_{l<\beta \in S'} S_{\beta l}\Big)W_{t,x,Y_{S'\cup\{l\}}}(\bm{\xi}_{S'\cup\{l\}})
\end{equation*}
because $\big(S_{\beta\alpha}\big)\big|_{\xi_\beta = 1} = \frac{q}{p}$. Also, noting that $\sum_A(A\setminus S') = \sum(A\setminus S') - <l>,$ we obtain that
\begin{equation*}
c_{S'} \times (q/p)^{<l>} =  q^{(n-1)(n-2)/2}\times \frac{q^{\sum(A\setminus S') - (n-1)|A\setminus S'|}}{p^{{\sum(A\setminus S') - |A\setminus S'|(|A\setminus S'| +1)/2}}}:= \tilde{c}_{S',l}.
\end{equation*}
Now, we compute
\begin{equation}\label{202am829}
\begin{aligned}
&\sum_{l=1}^n~\sum_{\substack{S' \subset \{1,\dots,n\} \setminus \{l\},\\ S' \neq \emptyset}}\tilde{c}_{S',l}\dashint_{c}\cdots \dashint_{c}\Big(\prod_{\substack{\alpha<\beta, \\\alpha,\beta \in S'}}T_{\beta\alpha}\Big)\Big(\prod_{l<\beta \in S'} S_{\beta l}\Big) J(\bm{\xi}_{S'}) W_{t,x,Y_{S'\cup\{l\}}}(\bm{\xi}_{S'\cup\{l\}}) d\bm{\xi}_{S'\cup \{l\}}.
\end{aligned}
\end{equation}
 Note that if we let $S = S' \cup \{l\}$, then  $\Sigma(A\setminus S') = \Sigma(S^c)$ and $|A\setminus S'| = |S^c|$. This implies that if a pair $(S',l)$ with
 $S' \subset \{1,\dots, n\}\setminus \{l\}$ and a pair $(S'',l')$ with $S'' \subset \{1,\dots, n\}\setminus \{l'\}$ satisfy $S' \cup \{l\} = S'' \cup \{l'\}$, then $\tilde{c}_{S',l} = \tilde{c}_{S'',l'}$. Hence, the sum (\ref{202am829}) is equivalent to
\begin{equation*}\label{203am829}
\begin{aligned}
&\sum_{|S| \geq 2}~\sum_{(S',l)}\tilde{c}_{S}\dashint_{c}\cdots \dashint_{c}\Big(\prod_{\substack{\alpha<\beta, \\\alpha,\beta \in S'}}T_{\beta\alpha}\Big)\Big(\prod_{l<\beta \in S'} S_{\beta l}\Big) J(\bm{\xi}_{S'})  W_{t,x,Y}(\bm{\xi}_{S}) d\bm{\xi}_{S}
\end{aligned}
\end{equation*}
where $\sum_{(S',l)}$ implies the sum over  partitions of $S = S' \cup \{l\}$  and
\begin{equation*}
\tilde{c}_{S}= q^{(n-1)(n-2)/2} \frac{q^{\Sigma(S^c) - (n-1)|S^c|}}{p^{{\Sigma(S^c) - |S^c|(|S^c| +1)/2}}}.
\end{equation*}
Replacing $N$ by $|S|$ and $m$ by $|S|-1$ in the identity (1.9) in \cite{Tracy-Widom-2008}, which is
\begin{equation*}
\sum_{\substack{|A| = m,\\ A\subset \{1,\dots,N\}}}\prod_{\substack{i \in A,\\j \in A^c}}\frac{p+q\xi_i\xi_j - \xi_i}{\xi_j - \xi_i}\cdot \Big(1- \prod_{j \in A^c}\xi_j\Big) = q^m \qbinom{N-1}{m}\Big(1- \prod_{j=1}^N\xi_j\Big),
\end{equation*}
and recalling the form of $T_{\beta\alpha}$ in (\ref{1100am911}),  we obtain
\begin{equation*}
\sum_{(S',\{l\})}\Big(\prod_{\substack{\alpha<\beta, \\\alpha,\beta \in S'}}T_{\beta\alpha}\Big)\Big(\prod_{l<\beta \in S'} S_{\beta l}\Big) J(\bm{\xi}_{S'}) =q^{|S|-1} \Big(\prod_{\substack{\alpha<\beta, \\\alpha,\beta \in S}}T_{\beta\alpha}\Big) I(\bm{\xi}_{S})
\end{equation*}
and a simple calculation shows that
\begin{equation}\label{344pm829}
\tilde{c}_{S} \times q^{|S|-1}  = q^{n(n-1)/2} \frac{q^{\Sigma(S^c) - n|S^c|}}{p^{{\Sigma(S^c) - |S^c|(|S^c| +1)/2}}} = c_S.
\end{equation}
Hence, we showed that  (\ref{202am829}) equals
\begin{equation*}
\sum_{|S| \geq 2} c_S\dashint_{c}\cdots \dashint_{c}  \Big(\prod_{\substack{\alpha<\beta, \\\alpha,\beta \in S}} T_{\beta\alpha}\Big) I(\bm{\xi}_{S}) W_{t,x,Y_S}(\bm{\xi}_{S}) d\bm{\xi}_{S}.
\end{equation*}
Finally,
\begin{equation*}
\sum_{l=1}^n\dashint_c\bigg(\Big(\prod_{l<\beta} S_{\beta l}\Big) W_{t,x,Y}(\bm{\xi})\bigg)\bigg|_{\xi_i = 1,\,i\neq l}d\xi_l  = \sum_{l=1}^n\Big(\frac{q}{p}\Big)^{n-l}\dashint_c W_{t,x,y_l}({\xi}_{y_l})d\xi_l,
\end{equation*}
in (\ref{141am829}), and we note that  $(q/p)^{n-l} =c_S$ when $S = \{l\}$. This completes the proof.
\end{proof}
\end{lemma}
\begin{lemma}\label{212am827}
For any $U = \{u_1,\dots, u_n\}\subset \{1,2,\dots \}$, (\ref{355am814}) is true for all $n\geq 2.$
\end{lemma}
\begin{proof}
It is clear that Lemma \ref{542pm828} and \ref{334am513}  imply Lemma \ref{212am827}.
\end{proof}
\begin{lemma}\label{517pm818}
Let $S\subset \{1,\dots, n\}$ with $|S|=m$. Then,
\begin{equation}\label{518pm830}
\sum_{\substack{S~\textrm{with}\\|S| = m}} \bigg(\prod_{\substack{\alpha<\beta,\\\alpha,\beta \in S^c}} T_{\beta\alpha}\bigg)\bigg(\prod_{\substack{\alpha<\beta,\\\alpha,\beta \in S}} T_{\beta\alpha}\bigg)\bigg( \prod_{\substack{\alpha<\beta,\\\ \alpha\in S, \beta \in S^c}} S_{\beta\alpha}\bigg) = \qbinom{n}{m}\bigg(\prod_{\substack{\alpha<\beta,\\ \alpha,\beta \in \{1,\dots, n\}} }T_{\beta\alpha}\bigg).
\end{equation}
\begin{proof}
Let us recall the forms of $S_{\beta\alpha}$ and $T_{\beta\alpha}$ in (\ref{1100am911}).  Dividing both sides of (\ref{518pm830}) by
\begin{equation*}
\bigg(\prod_{\substack{\alpha<\beta\\ \alpha,\beta \in \{1,\dots, n\}} }T_{\beta\alpha}\bigg),
\end{equation*}
we obtain
\begin{equation}\label{541pm830}
\begin{aligned}
&\sum_{\substack{S~\textrm{with}\\|S| = m}}\frac{\bigg( \displaystyle\prod_{\substack{\alpha<\beta,\\\ \alpha\in S, \beta \in S^c}} S_{\beta\alpha}\bigg)}{\bigg(\displaystyle\prod_{\substack{\alpha<\beta,\\\ \alpha\in S, \beta \in S^c} }T_{\beta\alpha}\bigg)\bigg(\displaystyle\prod_{\substack{\alpha>\beta,\\\ \alpha\in S, \beta \in S^c} }T_{\alpha\beta}\bigg)} =\sum_{\substack{S~\textrm{with}\\|S| = m}}\prod_{\substack{\alpha<\beta,\\ \alpha \in S,\,\beta \in S^c   }}
\bigg(\frac{p+q\xi_{\alpha}\xi_{\beta} - \xi_{\beta}}{\xi_{\alpha}-  \xi_{\beta}} \bigg)\prod_{\substack{\alpha>\beta, \\ \alpha \in S,\,\beta \in S^c  }}
\bigg(\frac{p+q\xi_{\alpha}\xi_{\beta} - \xi_{\beta}}{\xi_{\alpha}-  \xi_{\beta}} \bigg) \\
=~&\sum_{\substack{S^c~\textrm{with}\\|S^c| = n-m}} \displaystyle\prod_{\substack{\alpha \in S,\\\beta \in S^c  }}
\bigg(\frac{p+q\xi_{\alpha}\xi_{\beta} - \xi_{\beta}}{\xi_{\alpha}-  \xi_{\beta}} \bigg)=
 \qbinom{n}{n-m} = \qbinom{n}{m}
 \end{aligned}
\end{equation}
where the third equality follows from  the identity (6.3) in \cite{Tracy-Widom-2008}, which is
\begin{equation*}
\sum_{\substack{A\subset \{1,\dots,N\},\\
|A| = m}}\prod_{\substack{i\in A,\\j \in A^c}}\frac{p+q\xi_i\xi_j - \xi_i}{\xi_j - \xi_i} = \qbinom{N}{m}.
\end{equation*}
\end{proof}
\end{lemma}
Let us recall the Cauchy binomial theorem,
\begin{equation}\label{217am830}
\prod_{k=1}^n(1+y\tau^k) = \sum_{k=0}^ny^k\tau^{k(k+1)/2}\qbinom{n}{k}_{\tau}.
\end{equation}
where $\qbinom{n}{k}_{\tau}$ implies the usual q-Binomial Coefficient.
Setting $y=-1$ and $\tau = p/q$ in (\ref{217am830}), we obtain the following identity.
\begin{lemma}\label{118am92}
For each $l=1,2,\dots$,
\begin{equation}\label{459pm822}
\prod_{i=1}^{l-1}(q^i - p^i) = \sum_{k=0}^{l-1}(-1)^k\qbinom{l-1}{k}p^{k(k+1)/2}q^{(l-k)(l-k-1)/2}.
\end{equation}
\end{lemma}
\section{Formula for $N=3$}\label{1249am830}
In this section, we find $\mathbb{P}(\eta(t) = x)$ for $N=3$. The methods used in the section are further generalized in the next section for general $N$-particle systems.
\subsection{Decomposition of the transition probability}
For $N=3$, it is straightforward that
\begin{equation}\label{210am527}
 \mathbb{P}(\eta(t) = x) = \sum_{\substack{X\,\textrm{with} \\x_3 = x}} P_{Y}(X,\nu_3;t) +  \sum_{\substack{X\,\textrm{with} \\x_2 = x}} P_{Y}(X,\nu_2;t) + \sum_{\substack{X\,\textrm{with} \\x_1 = x}} P_{Y}(X,\nu_1;t),
\end{equation}
and
\begin{equation}\label{722pm812}
P_{Y}(X,\nu_n;t)= \sum_{\sigma\in \mathcal{S}_3}~\dashint_c \dashint_c \dashint_c~[\mathbf{A}_{\sigma}]_{\nu_n,\,221}\prod_{i=1}^3\Big(\xi_{\sigma(i)}^{x_{i}-y_{\sigma(i)}-1}e^{(\frac{p}{\xi_i}+q\xi_i - 1 )t}\Big)~ d\xi_1d\xi_2 d\xi_3
\end{equation}
where the matrix elements  $[\mathbf{A}_{\sigma}]_{\nu_n,\,221}$ are given as in Table \ref{fig:figure1}, which can be obtained by the formulas of $[\mathbf{A}_{\sigma}]_{\nu_n,\,\nu_N}$ in \cite{Lee-Raimbekov-2022}.
\begin{table}[H]
\[\begin{tabular}{ | c || c || c ||  c |}
\hline
  $\sigma$  &$[\mathbf{A}_{\sigma}]_{221,221}$       & $[\mathbf{A}_{\sigma}]_{212,221}$      & $[\mathbf{A}_{\sigma}]_{122,221}$      \\
\hline
123 & 1                     &           0            &       0                \\

213 & $S_{21}$              &         0              &        0               \\

132 & $Q_{32}$              & $pT_{32}$             &          0             \\

231 & $S_{21} Q_{31}$        & $pT_{31} S_{21}$       &        0               \\

312 & $S_{31} Q_{32}$        & $pT_{32} Q_{31}$       & $pT_{31} pT_{32}$      \\

321 & $S_{21} S_{32} Q_{31}$  & $pT_{31} S_{21} Q_{32}$ & $pT_{32} pT_{31} S_{21}$\\
\hline

\end{tabular}\]
 \caption{$[\mathbf{A}_{\sigma}]_{\nu_n,221}$ for $N=3$}
    \label{fig:figure1}
\end{table}
Using that
\begin{equation}\label{1119pm527}
Q_{\beta\alpha} = S_{\beta\alpha}  - pT_{\beta\alpha},
\end{equation}
we write $[\mathbf{A}_{\sigma}]_{\nu_n,\,221}$ as
\begin{equation*}
[\mathbf{A}_{\sigma}]_{\nu_n,\,221} = [\mathbf{A}_{\sigma}]_{\nu_n,\,221}^+ + [\mathbf{A}_{\sigma}]_{\nu_n,\,221}^-
\end{equation*}
where $[\mathbf{A}_{\sigma}]_{\nu_n,\,221}^+$ and $[\mathbf{A}_{\sigma}]_{\nu_n,\,221}^-$ are given in Table \ref{table2} below.
 \begin{table}[H]
    \[\begin{tabular}{|c||c|c||c|c||c|}
    \hline
     $\sigma$ & $[\mathbf{A}_{\sigma}]^{+}_{221,221}$ & $[\mathbf{A}_{\sigma}]^{-}_{221,221}$ & $[\mathbf{A}_{\sigma}]^{+}_{212,221}$ & $[\mathbf{A}_{\sigma}]^{-}_{212,221}$ & $[\mathbf{A}_{\sigma}]_{122,221}$    \\ \hline
    $123$ & $1$ & 0 &0 & 0 & 0  \\
    $213$ & $S_{21}$ & 0 &0 & 0&0   \\
    $132$ & $S_{32}$ & $-pT_{32}$ & $pT_{32}$ &0 &  0  \\
    $231$ & $S_{21}S_{31}$ & $-pT_{31}S_{21}$ & $pT_{31}S_{21}$ & 0 & 0  \\
    $312$ & $S_{31}S_{32}$ & $-pT_{32}S_{31}$ & $pT_{32}S_{31}$ & $-pT_{31}pT_{32}$ & $pT_{31}pT_{32}$ \\
    $321$ & $S_{21}S_{32}S_{31}$ & $-pT_{31}S_{21}S_{32}$ & $pT_{31}S_{21}S_{32}$ & $-pT_{32}pT_{31}S_{21}$ & $pT_{32}pT_{31}S_{21}$  \\
    \hline
    \end{tabular}\]
    \caption{$[\mathbf{A}_{\sigma}]^{\pm}_{\nu_n,221}$ for $N=3$}
    \label{table2}
\end{table}
Following the convention above, we decompose the transition probabilities $P_{Y}(X,\nu_n;t)$ as follows:
 \begin{equation*}
 P_{Y}(X,\nu_n;t) =  P_{Y}^+(X,\nu_n;t) + P_{Y}^-(X,\nu_n;t)
 \end{equation*} where $P_{Y}^+(X,\nu_n;t)$ and $P_{Y}^-(X,\nu_n;t)$ are the \textit{components} of $P_{Y}(X,\nu_n;t)$ obtained by replacing $[\mathbf{A}_{\sigma}]_{\nu_n,\,221}$ by $[\mathbf{A}_{\sigma}]_{\nu_n,\,221}^+$ and $[\mathbf{A}_{\sigma}]_{\nu_n,\,221}^-$, respectively, in (\ref{722pm812}). With these settings,  we will compute the sum (\ref{210am527}) in the following manner:
 \begin{equation}\label{21011am527}
 \begin{aligned}
 \mathbb{P}(\eta(t) = x) =~& \underbrace{\sum_{\substack{X\,\textrm{with} \\x_3 = x}} P_{Y}^+(X,\nu_3;t)}_{(a)} + \underbrace{\bigg(\sum_{\substack{X\,\textrm{with} \\x_3 = x}} P_{Y}^-(X,\nu_3;t) +  \sum_{\substack{X\,\textrm{with} \\x_2 = x}} P_{Y}^+(X,\nu_2;t)\bigg)}_{(b)} \\[10pt]
 &\hspace{1cm} +~ \underbrace{\bigg( \sum_{\substack{X\,\textrm{with} \\x_2 = x}} P_{Y}^-(X,\nu_2;t)+ \sum_{\substack{X\,\textrm{with} \\x_1 = x}} P_{Y}(X,\nu_1;t)\bigg)}_{(c)}.
 \end{aligned}
\end{equation}
\subsection{Computation of (a)  in (\ref{21011am527})}
We observe that $[\mathbf{A}_{\sigma}]_{221,221}^+$ is the same as $A_{\sigma}$ in (\ref{126am830}) for the single-species ASEP. Hence, the sum $(a)$ in (\ref{21011am527}) is the same as the probability distribution of the rightmost particle's position in the single-species ASEP. Using Theorem 5.1 with $N=m=3$ in \cite{Tracy-Widom-2008}, we obtain (a) equal to
\begin{equation}\label{1251am528}
\begin{aligned}
&  q^3 \dashint_{c}\dashint_{c}\dashint_{c}T_{32}T_{31}T_{21}I(\xi_1,\xi_2,\xi_3)W_{t,x,y_1,y_2,y_3}(\xi_{1},\xi_2, \xi_{3})d\xi_{1}\xi_2 d\xi_{3} \\[8pt]
&~~+ \frac{q^3}{p^2} \,\dashint_{c}\dashint_{c}T_{21}I(\xi_1,\xi_2)W_{t,x,y_1,y_2}(\xi_{1},\xi_2)d\xi_{1}\xi_2\\[8pt]
&~~+ \frac{q^2}{p}\, \dashint_{c}\dashint_{c}T_{31}I(\xi_1,\xi_3)W_{t,x,y_1,y_3}(\xi_{1},\xi_3)d\xi_{1}\xi_3 \\[8pt]
&~~+ {q}\, \dashint_{c}\dashint_{c}T_{32}I(\xi_2,\xi_3)W_{t,x,y_2,y_3}(\xi_{2},\xi_3)d\xi_{2}\xi_3 \\[8pt]
&~~+ \frac{q^2}{p^2}\, \dashint_{c}W_{t,x,y_1}(\xi_{1})d\xi_{1}~ +~\frac{q}{p}\, \dashint_{C_{r}}W_{t,x,y_2}(\xi_{2})d\xi_{2} ~ + ~ \dashint_{c}W_{t,x,y_3}(\xi_{3})d\xi_{3}.
\end{aligned}
\end{equation}
\subsection{Computation of (b) in (\ref{21011am527})}
Recall that $P_{Y}^-(X,\nu_3;t)$ and $P_{Y}^+(X,\nu_2;t)$ are in the form of
\begin{equation*}
\sum_{\sigma \in \mathcal{S}_3}~\dashint_c\dashint_c\dashint_c ~(~\cdots~).
\end{equation*}
We decompose $\sum_{\sigma \in \mathcal{S}_3}$ as in
\begin{equation*}
\sum_{\sigma \in \mathcal{S}_3}(~\cdots~)= \sum_{l=1}^3\sum_{\substack{\sigma\,\textrm{with}\\
\sigma(3)=l}}(~\cdots~)
\end{equation*}
and evaluate the sum (b) in the following manner:
\begin{equation*}
\begin{aligned}
(b)~=~&\sum_{l=1}^3\bigg(\underbrace{\sum_{\substack{\sigma\,\textrm{with}\\
\sigma(3)=l}}\sum_{\substack{X\,\textrm{with} \\x_3 = x}}\dashint_c\dashint_c\dashint_c [\mathbf{A}_{\sigma}]_{221,\,221}^- \prod_{i=1}^3\Big(\xi_{\sigma(i)}^{x_i - y_{\sigma(i)}-1} e^{(\frac{p}{\xi_i}+q\xi_i - 1)t}\Big)d\xi_1d\xi_2d\xi_3 }_{(*)}\\
&\hspace{1cm}+~
\underbrace{\sum_{\substack{\sigma\,\textrm{with}\\
\sigma(3)=l}}\sum_{\substack{X\,\textrm{with} \\x_2 = x}}\dashint_c\dashint_c\dashint_c [\mathbf{A}_{\sigma}]_{212,\,221}^+ \prod_{i=1}^3\Big(\xi_{\sigma(i)}^{x_i - y_{\sigma(i)}-1} e^{(\frac{p}{\xi_i}+q\xi_i - 1)t}\Big)d\xi_1d\xi_2d\xi_3}_{(**)}\bigg).
\end{aligned}
\end{equation*}
Note that when $l=3$, both $(*)$ and $(**)$ are zero because $[\mathbf{A}_{\sigma}]_{221,221}^-=[\mathbf{A}_{\sigma}]_{212,221}^+=0$ in this case (see Table \ref{table2}).  When $l=2$, $(*)$ can be written
\begin{equation}\label{1251am5289}
\begin{aligned}
&-\sum_{z_1,z_2=1}^{\infty}\dashint_c\bigg(\,\dashint_c\dashint_c \Big(\big(\xi_1\xi_3\big)^{-z_1}\big(\xi_1\big)^{-z_2}+S_{31}\big(\xi_3\xi_1\big)^{-z_1}\big(\xi_3\big)^{-z_2}\Big)pT_{32} \\
&\hspace{4cm}\times W_{t,x,y_1,y_2,y_3}(\xi_1,\xi_2,\xi_3)d\xi_1d\xi_3\bigg)d\xi_2.
\end{aligned}
\end{equation}
If we apply  Lemma \ref{212am827} with $U=\{1,3\}$ and \textit{Remark} \ref{704pm913} to the sum of double integral with respect to $\xi_1,\xi_3$ in (\ref{1251am5289}), then (\ref{1251am5289}) equals
\begin{equation}\label{1251am5281}
\begin{aligned}
& - qp\, \dashint_{c}\dashint_{c}\dashint_{c}T_{31}T_{32}J(\xi_1,\xi_3)W_{t,x,y_1,y_2,y_3}(\xi_{1},\xi_2, \xi_{3})d\xi_{1}\xi_2 d\xi_{3} \\[8pt]
&- \frac{q}{p} \,\dashint_{c}\dashint_{c}J(\xi_1)W_{t,x,y_1,y_2}(\xi_{1},\xi_2)d\xi_{1}\xi_2
~-p\dashint_{c}\dashint_{c}T_{32}J(\xi_3)W_{t,x,y_2,y_3}(\xi_{2},\xi_3)d\xi_{2}\xi_3 \\[8pt]
& - \dashint_{c}W_{t,x,y_2}(\xi_{2})d\xi_{2}.
\end{aligned}
\end{equation}
When $l=2$, $(**)$ equals
\begin{equation}\label{1251am52899}
\begin{aligned}
&~\sum_{\substack{\sigma\,\textrm{with}\\
\sigma(3)=2}}\sum_{z,v=1}^{\infty}\dashint_c\dashint_c\dashint_c [\mathbf{A}_{\sigma}]_{212,\,221}^+ \xi_{\sigma(1)}^{-z}\xi_{\sigma(3)}^{v}W_{t,x,y_1,y_2,y_3}(\xi_1,\xi_2,\xi_3)d\xi_1d\xi_2d\xi_3 \\
=&~\sum_{\substack{\sigma\,\textrm{with}\\
\sigma(3)=2}}\sum_{z=1}^{\infty}\dashint_c\dashint_c\dashint_c [\mathbf{A}_{\sigma}]_{212,\,221}^+ \xi_{\sigma(1)}^{-z}\frac{\xi_2}{1-\xi_2}W_{t,x,y_1,y_2,y_3}(\xi_1,\xi_2,\xi_3)d\xi_1d\xi_2d\xi_3\\
=& ~\sum_{z=1}^{\infty}\,\dashint_c\bigg(\,\dashint_c\dashint_c \Big(\big(\xi_1\big)^{-z}+S_{31}\big(\xi_3\big)^{-z}\Big)pT_{32}W_{t,x,y_1,y_2,y_3}(\xi_1,\xi_2,\xi_3)d\xi_1d\xi_3\bigg)\frac{\xi_2}{1-\xi_2}d\xi_2.
\end{aligned}
\end{equation}
If we apply Lemma \ref{334am513} with $U=\{1,3\}$ and \textit{Remark} \ref{704pm913} to the sum of the double integral with respect to $\xi_1,\xi_3$ in (\ref{1251am52899}), then (\ref{1251am52899}) equals
\begin{equation}\label{am5282}
\begin{aligned}
&  qp\, \dashint_{c}\dashint_{c}\dashint_{c}T_{31}T_{32}I(\xi_1,\xi_3)\frac{\xi_2}{1-\xi_2}W_{t,x,y_1,y_2,y_3}(\xi_{1},\xi_2, \xi_{3})d\xi_{1}\xi_2 d\xi_{3} \\[8pt]
&+ \frac{q}{p} \,\dashint_{c}\dashint_{c}\frac{\xi_2}{1-\xi_2}W_{t,x,y_1,y_2}(\xi_{1},\xi_2)d\xi_{1}\xi_2~+p~ \dashint_{c}\dashint_{c}T_{32}\frac{\xi_2}{1-\xi_2}W_{t,x,y_2,y_3}(\xi_{2},\xi_3)d\xi_{2}\xi_3.
\end{aligned}
\end{equation}
Summing (\ref{1251am5281}) and (\ref{am5282}), we obtain
\begin{equation}\label{1251am5284}
\begin{aligned}
& - qp\, \dashint_{c}\dashint_{c}\dashint_{c}T_{31}T_{32}I(\xi_1,\xi_2,\xi_3)W_{t,x,y_1,y_2,y_3}(\xi_{1},\xi_2, \xi_{3})d\xi_{1}\xi_2 d\xi_{3} \\[8pt]
&- \frac{q}{p} \,\dashint_{c}\dashint_{c}I(\xi_1,\xi_2)W_{t,x,y_1,y_2}(\xi_{1},\xi_2)d\xi_{1}\xi_2\\[8pt]
&-p\, \dashint_{c}\dashint_{c}T_{32}I(\xi_2,\xi_3)W_{t,x,y_2,y_3}(\xi_{2},\xi_3)d\xi_{2}\xi_3 ~ - \dashint_{c}W_{t,x,y_2}(\xi_{2})d\xi_{2}.
\end{aligned}
\end{equation}
Similarly, when $\sigma(3) =1$, we obtain
\begin{equation}\label{608am5284}
\begin{aligned}
 (*) + (**)~~ = &~ -~ qp\, \dashint_{c}\dashint_{c}\dashint_{c}T_{31}T_{32}S_{21}I(\xi_1,\xi_2,\xi_3)W_{t,x,y_1,y_2,y_3}(\xi_{1},\xi_2, \xi_{3})d\xi_{1}\xi_2 d\xi_{3} \\[8pt]
&~-~ \frac{q}{p} \,\dashint_{c}\dashint_{c}S_{21}I(\xi_1,\xi_2)W_{t,x,y_1,y_2}(\xi_{1},\xi_2)d\xi_{1}\xi_2\\[8pt]
&~-~q\, \dashint_{c}\dashint_{c}T_{31}I(\xi_1,\xi_3)W_{t,x,y_1,y_3}(\xi_{1},\xi_3)d\xi_{1}\xi_3 ~-~ \frac{q}{p}\dashint_{c}W_{t,x,y_1}(\xi_{1})d\xi_{1}.
\end{aligned}
\end{equation}
Finally, using the fact that
\begin{equation*}
1+S_{\beta\alpha} = T_{\beta\alpha},
\end{equation*}
we find that the sum of (\ref{1251am5284}) and (\ref{608am5284}), that is,  $(b)$ equals
\begin{equation}\label{608am52844}
\begin{aligned}
& - qp\, \dashint_{c}\dashint_{c}\dashint_{c}T_{31}T_{32}T_{21}I(\xi_1,\xi_2,\xi_3)W_{t,x,y_1,y_2,y_3}(\xi_{1},\xi_2, \xi_{3})d\xi_{1}\xi_2 d\xi_{3} \\[6pt]
&- \frac{q}{p} \,\dashint_{c}\dashint_{c}T_{21}I(\xi_1,\xi_2)W_{t,x,y_1,y_2}(\xi_{1},\xi_2)d\xi_{1}\xi_2  -q\, \dashint_{c}\dashint_{c}T_{31}I(\xi_1,\xi_3)W_{t,x,y_1,y_3}(\xi_{1},\xi_3)d\xi_{1}\xi_3 \\[6pt]
&-p\, \dashint_{c}\dashint_{c}T_{32}I(\xi_2,\xi_3)W_{t,x,y_2,y_3}(\xi_{2},\xi_3)d\xi_{2}\xi_3 \\[6pt]
& - \frac{q}{p}\,\dashint_{c}I(\xi_1)W_{t,x,y_1}(\xi_{1})d\xi_{1}  -~\dashint_{c}I(\xi_2)W_{t,x,y_2}(\xi_{2})d\xi_{2}.
\end{aligned}
\end{equation}
\subsection{Computation of $(c)$ in (\ref{21011am527})}
It can be easily obtained that the first sum in (c) equals
\begin{equation*}
\begin{aligned}
~& p^3\dashint_c\dashint_c\dashint_c T_{21}T_{31}T_{32}\frac{\xi_1\xi_2-1}{(\xi_1-1)(\xi_2-1)(\xi_3-1)}W_{t,x,y_1,y_2,y_3}(\xi_1,\xi_2,\xi_3)d\xi_1d\xi_2d\xi_3 \\[5pt]
&~~~+~p\dashint_c\dashint_c T_{21}I(\xi_1,\xi_2)W_{t,x,y_1,y_2}(\xi_1,\xi_2)d\xi_1d\xi_2
\end{aligned}
\end{equation*}
and the second sum in (c) equals
\begin{equation*}
p^3\dashint\dashint\dashint T_{21}T_{31}T_{32}\frac{\xi_1\xi_2}{(1-\xi_1)(1-\xi_2)}W_{t,x,y_1,y_2,y_3}(\xi_1,\xi_2,\xi_3)d\xi_1d\xi_2d\xi_3.
\end{equation*}
Summing these two results, we can see that $(c)$ is equal to
\begin{equation}\label{242am86}
\begin{aligned}
&p^3\dashint_c\dashint_c\dashint_c T_{21}T_{31}T_{32}I(\xi_1,\xi_2,\xi_3)W_{t,x,y_1,y_2,y_3}(\xi_1,\xi_2,\xi_3)d\xi_1d\xi_2d\xi_3 \\[5pt]
&~~~+~p\dashint_c\dashint_c T_{21}I(\xi_1,\xi_2)W_{t,x,y_1,y_2}(\xi_1,\xi_2)d\xi_1d\xi_2.
\end{aligned}
\end{equation}
\subsection{$(a) + (b) +(c)$ in (\ref{21011am527})}
We observe that the integrals of the same variables in (\ref{1251am528}), (\ref{608am52844}) and (\ref{242am86})  have the same integrand, so their sum will be the integral multiplied by the sum of coefficients of those integrals. Moreover, the sum of these coefficients  can be factorized. The final form of $(a) + (b) +(c)$ is as follows:
\begin{equation}\label{304am86}
\begin{aligned}
& \mathbb{P}(\eta(t) = x) = (q-p)^2 \dashint_{c}\dashint_{c}\dashint_{c}I(\xi_1,\xi_2,\xi_3)W_{t,x,y_1,y_2,y_3}(\xi_{1},\xi_2, \xi_{3})d\xi_{1}\xi_2 d\xi_{3} \\[8pt]
&+ \frac{(q-p)^2}{p^2} \,\dashint_{c}\dashint_{c}I(\xi_1,\xi_2)W_{t,x,y_1,y_2}(\xi_{1},\xi_2)d\xi_{1}\xi_2+ \frac{q}{p}(q-p)\, \dashint_{c}\dashint_{c}I(\xi_1,\xi_3)W_{t,x,y_1,y_3}(\xi_{1},\xi_3)d\xi_{1}\xi_3 \\[8pt]
&+ {(q-p)}\, \dashint_{c}\dashint_{c}I(\xi_2,\xi_3)W_{t,x,y_2,y_3}(\xi_{2},\xi_3)d\xi_{2}d\xi_3 + \frac{q}{p^2}(q-p)\, \dashint_{c}I(\xi_1)W_{t,x,y_1}(\xi_{1})d\xi_{1}\\[8pt]
&+\frac{(q-p)}{p}\, \dashint_{c}I(\xi_2)W_{t,x,y_2}(\xi_{2})d\xi_{2}+ \dashint_{c}I(\xi_3)W_{t,x,y_3}(\xi_{3})d\xi_{3}.
\end{aligned}
\end{equation}
\section{Proof of Theorem \ref{102am830}}\label{101am830}
In this section, we extend the idea used in Section \ref{1249am830} to prove Theorem \ref{102am830}. For $[\mathbf{A}_{\sigma}]_{\nu_n,\,\nu_N}$ containing $Q_{\beta\alpha}$, let denote such $[\mathbf{A}_{\sigma}]_{\nu_n,\,\nu_N}$ by $[\mathbf{A}_{\sigma}]_{\nu_n,\,\nu_N}^+$ if $Q_{\beta\alpha}$ is replaced by $S_{\beta\alpha}$, and by $[\mathbf{A}_{\sigma}]_{\nu_n,\,\nu_N}^-$ if $Q_{\beta\alpha}$ is replaced by $-pT_{\beta\alpha}$. For $[\mathbf{A}_{\sigma}]_{\nu_n,\,\nu_N}$ containing no $Q_{\beta\alpha}$, we set $[\mathbf{A}_{\sigma}]_{\nu_n,\,\nu_N} = [\mathbf{A}_{\sigma}]^+_{\nu_n,\,\nu_N}$ and $[\mathbf{A}_{\sigma}]^-_{\nu_n,\,\nu_N} = 0$. Then, we have an expression
\begin{equation*}
[\mathbf{A}_{\sigma}]_{\nu_n,\,\nu_N} = [\mathbf{A}_{\sigma}]_{\nu_n,\,\nu_N}^+ + [\mathbf{A}_{\sigma}]_{\nu_n,\,\nu_N}^-
\end{equation*}
because $Q_{\beta\alpha} = S_{\beta\alpha} - pT_{\beta\alpha}$. See Table \ref{fig:figure1} and Table \ref{table2} for $N=3$, for example. Following the convention above, we decompose the transition probabilities $P_{Y}(X,\nu_n;t)$ as follows:
 \begin{equation*}
 P_{Y}(X,\nu_n;t) =  P_{Y}^+(X,\nu_n;t) + P_{Y}^-(X,\nu_n;t)
 \end{equation*} where $P_{Y}^+(X,\nu_n;t)$ and $P_{Y}^-(X,\nu_n;t)$ are the \textit{components} of $P_{Y}(X,\nu_n;t)$ obtained by replacing $[\mathbf{A}_{\sigma}]_{\nu_n,\,\nu_N}$ by $[\mathbf{A}_{\sigma}]_{\nu_n,\,\nu_N}^+$ and $[\mathbf{A}_{\sigma}]_{\nu_n,\,\nu_N}^-$, respectively, in (\ref{533pm813}). Hence, (\ref{1051pm527}) is written as
\begin{equation}\label{415am87}
 \begin{aligned}
 \mathbb{P}(\eta(t) = x) =~& \sum_{\substack{X\,\textrm{with} \\x_N = x}} P_{Y}^+(X,\nu_N;t) + \bigg(\sum_{\substack{X\,\textrm{with} \\x_N = x}} P_{Y}^-(X,\nu_N;t) +  \sum_{\substack{X\,\textrm{with} \\x_{N-1} = x}} P_{Y}^+(X,\nu_{N-1};t)\bigg) \\[10pt]
 &\hspace{1cm} +~ \cdots ~+~\bigg( \sum_{\substack{X\,\textrm{with} \\x_2 = x}} P_{Y}^-(X,\nu_2;t)+ \sum_{\substack{X\,\textrm{with} \\x_1 = x}} P_{Y}(X,\nu_1;t)\bigg).
 \end{aligned}
\end{equation}
We will find the general formula of
\begin{equation}\label{751pm87}
\bigg(\sum_{\substack{X\,\textrm{with} \\x_{n+1} = x}} P_{Y}^-(X,\nu_{n+1};t) +  \sum_{\substack{X\,\textrm{with} \\x_{n} = x}} P_{Y}^+(X,\nu_{n};t)\bigg)
\end{equation}
for $n=1,\dots, N$ with conventions $P_{Y}^-(X,\nu_{N+1};t) =0$ and $ P_{Y}(X,\nu_1;t) =  P_{Y}^+(X,\nu_1;t)$, and then sum the result over $n=1,\dots,N$.
\begin{proposition}\label{1238am829}
For $n=1,\dots, N$,
\begin{equation}\label{331am814}
\begin{aligned}
&\sum_{\substack{X\,\textrm{with} \\x_{n+1} = x}} P_{Y}^-(X,\nu_{n+1};t) +  \sum_{\substack{X\,\textrm{with} \\x_{n} = x}} P_{Y}^+(X,\nu_{n};t) = \\
&\hspace{2.5cm}\sum_{\substack{\emptyset \neq S \subset \{1,\dots, N\},\\
|S|\geq N-n}}\tilde{c}_S\dashint_c\cdots \dashint_c \bigg(\prod_{\substack{\alpha<\beta, \\ \alpha,\beta \in S}}T_{\beta\alpha}\bigg) I(\bm{\xi}_S)W_{t,x,Y_S}(\bm{\xi}_S) d\bm{\xi}_S
\end{aligned}
\end{equation}
where
\begin{equation*}
 \tilde{c}_{S} =
 \begin{cases}
 \displaystyle(-1)^{N-n}q^{n(n-1)/2}p^{(N-n)(N-n+1)/2} \frac{q^{\sum(S^c) - n|S^c|}}{p^{{\sum(S^c) - |S^c|(|S^c| +1)/2}}}\qbinom{|S|-1}{N-n}&~~\textrm{if}~N \in S, \\[20pt]
 \displaystyle(-1)^{N-n}q^{n(n-1)/2}p^{(N-n)(N-n-1)/2} \frac{q^{\sum(S^c) - n|S^c|-(N-n)}}{p^{{\sum(S^c) - |S^c|(|S^c| +1)/2-(N-n)}}}\qbinom{|S|}{N-n}&~~\textrm{if}~N \notin S.
 \end{cases}
 \end{equation*}
 \end{proposition}
\begin{proof}
First, we consider the case of $n=N$.
Since $[\mathbf{A}_{\sigma}]^+_{\nu_N,\nu_N}$ is the same as $A_{\sigma}$ in (\ref{126am830}) for the single-species ASEP, $\sum_{X\,\textrm{with}\, x_N = x} P_{Y}^+(X,\nu_N;t)$
is the same as the probability distribution of the rightmost particle's position at time $t$ in the single-species ASEP. Hence, when $n=N$, (\ref{331am814}) is satisfied
by (5.1) with $m=N$ in \cite[Theorem 5.1]{Tracy-Widom-2008}. Suppose that $n<N$. According to Theorem 1.6 in \cite{Lee-Raimbekov-2022},  $[\mathbf{A}_{\sigma}]_{\nu_n,\nu_N} = 0$ if $\sigma^{-1}(N) >n$, so the components of the transition probabilities are written
\begin{equation}\label{324am814}
P_Y^{\pm}(X,\nu_n;t) = \sum_{\substack{\sigma ~\textrm{with}\\ \sigma^{-1}(N) \leq n}}\dashint_c\cdots \dashint_c [\mathbf{A}_{\sigma}]_{\nu_n,\nu_N}^{\pm} \prod_{i=1}^N\Big(\xi_{\sigma(i)}^{x_i - y_{\sigma(i)}-1} e^{(\frac{p}{\xi_i}+q\xi_i - 1)t}\Big)d\bm{\xi}.
\end{equation}
The sum  in (\ref{324am814}) is written
\begin{equation}\label{549pm88}
 \sum_{\substack{\textrm{all}\,A\subset \{1,\dots,N\}\\
 \textrm{with}\,|A|=n,\,N \in A}} ~\sum_{\substack{\sigma\,\textrm{with}\\ \sigma(1),\dots, \sigma(n) \in A}} ~~ \sum_{\substack{\sigma\,\textrm{with}\\ \sigma(n+1),\dots, \sigma(N) \in A^c, \\ \sigma(1),\dots,\sigma(n)\,\textrm{fixed}}},
\end{equation}
and using this, we evaluate the sum on the left-hand side of (\ref{331am814}) as in the following manner:
\begin{itemize}
\item [(\rmnum{1})] \textbf{Step 1.} In this step, we compute the second sum on  the left-hand side of (\ref{331am814}). The second sum is written
\begin{equation*}
 \sum_{\substack{\textrm{all}\,A\subset \{1,\dots,N\}\\
\textrm{with}\, |A|=n,\,N \in A}}\bigg(\underbrace{\sum_{\substack{X\,\textrm{with} \\x_n = x}}~\sum_{\substack{\sigma\,\textrm{with}\\ \sigma(1),\dots, \sigma(n) \in A}} ~ \sum_{\substack{\sigma\,\textrm{with}\\ \sigma(n+1),\dots, \sigma(N) \in A^c\\ \sigma(1),\dots,\sigma(n)\,\textrm{fixed}}} \dashint_c\cdots \dashint_c  [\mathbf{A}_{\sigma}]_{\nu_n,\nu_N}^{+}\Big(\xi_{\sigma(i)}^{x_i - y_{\sigma(i)}-1} e^{(\frac{p}{\xi_i}+q\xi_i - 1)t}\Big)d\bm{\xi}}_{(\#)}\bigg).
\end{equation*}
We compute $(\#)$ for a fixed $A \subset \{1,\dots, N\}$ with $|A| = n$ and $N \in A$.
Set $x_n = x$ and
\begin{equation*}
\begin{aligned}
& x_{n-1} = x-z_{n-1}, x_{n-2} = x-z_{n-1}-z_{n-2},\dots, x_1 = x-z_{n-1}-\cdots -z_1, \\
& x_{n+1} =x+v_1, x_{n+2} = x+v_1+v_2,\dots, x_N = x+v_1+\cdots + v_{N-n}
\end{aligned}
\end{equation*}
where $z_i$ and $v_i$ are positive integers. Then,  $(\#)$   is written
\begin{equation}\label{909pm88}
\begin{aligned}
&\sum_{z_1,\dots,z_{n-1}}~\sum_{v_1,\dots,v_{N-n}} ~~\sum_{\substack{\sigma\,\textrm{with}\\ \sigma(1),\dots, \sigma(n) \in A}} ~ \sum_{\substack{\sigma\,\textrm{with}\\ \sigma(n+1),\dots, \sigma(N) \in A^c\\ \sigma(1),\dots,\sigma(n)\,\textrm{fixed}}}  \\[8pt]
&\hspace{0.5cm}\dashint_c\cdots \dashint_c [\mathbf{A}_{\sigma}]^+_{\nu_n,\nu_N}\prod_{i=1}^{n-1}\big(\xi_{\sigma(1)}\cdots\xi_{\sigma(i)}\big)^{-z_i}\prod_{i=1}^{N-n} \big(\xi_{\sigma(n+i)}\cdots\xi_{\sigma(N)}\big)^{v_i}W_{t,x,Y}(\bm{\xi})d\bm{\xi}.
\end{aligned}
\end{equation}
By Theorem 1.7, Proposition 1.8 in \cite{Lee-Raimbekov-2022} and (\ref{1119pm527}),  if  $\sigma(1),\dots, \sigma(n) \in A$ and $N \in A,$
\begin{equation}\label{224am810}
\begin{aligned}
[\mathbf{A}_{\sigma}]_{\nu_{n},\nu_N}^{+} =~& \Big(\prod_{\alpha\in A^c}pT_{N\alpha}\Big)\times \bigg(\prod_{\substack{\alpha<\beta \neq N\\
\beta \in A, \,\alpha \in A^c}} S_{\beta\alpha}\bigg) \\
&\times \bigg(\prod_{\substack{\textrm{all inversions $(\beta,\alpha)$ } \\
\textrm{in}~\sigma(1)\cdots \sigma(n)}} S_{\beta\alpha}\bigg)~ \times ~  \bigg(\prod_{\substack{\textrm{all inversions $(\beta,\alpha)$ } \\
\textrm{in}~\sigma(n+1)\cdots \sigma(N)}} S_{\beta\alpha}\bigg).
\end{aligned}
\end{equation}
Let us sum  the integral in (\ref{909pm88}) over $v_1,\dots, v_{N-n}$ and then over all permutations $\sigma$ with $\sigma(n+1),\dots, \sigma(N) \in A^c$ and $\sigma(1),\dots, \sigma(n)$ fixed. Then, (\ref{909pm88}) becomes
\begin{equation}\label{1012pm87}
\begin{aligned}
&\sum_{z_1,\dots, z_{n-1}}\sum_{\substack{\sigma\,\textrm{with}\\ \sigma(1),\dots, \sigma(n) \in A}} \dashint_c\cdots \dashint_c f(\bm{\xi})\bigg(\prod_{\substack{\textrm{all inversions $(\beta,\alpha)$ } \\
\textrm{in}~\sigma(1)\cdots \sigma(n)}} S_{\beta\alpha}\bigg) \prod_{i=1}^{n-1}\big(\xi_{\sigma(1)}\cdots\xi_{\sigma(i)}\big)^{-z_i}  d\bm{\xi}
\end{aligned}
\end{equation}
where
\begin{equation*}\label{418pm810}
\begin{aligned}
& f(\bm{\xi}) = p^{(N-n)(N-n-1)/2} \Big(\prod_{\alpha\in A^c}pT_{N\alpha}\Big) \bigg(\prod_{\substack{\alpha<\beta \neq N\\
\beta \in A, \,\alpha \in A^c}} S_{\beta\alpha}\bigg) \\
&\hspace{4cm}  \times   \bigg(\prod_{\substack{\alpha<\beta \\\alpha,\beta \in A^c }} T_{\beta\alpha}\bigg) \frac{ \prod_{a \in A^c}\xi_a}{\prod_{a\in A^c}(1-\xi_{a})} W_{t,x,Y}(\bm{\xi})
\end{aligned}
\end{equation*}
by using the identity (1.6) in \cite{Tracy-Widom-2008}. Here, we note that $f(\bm{\xi})$ can be made analytic inside a sufficiently large circle centered at the origin except at the origin for variable $\xi_i$ where $i \in A$ for the use of Lemma \ref{334am513} (see Remark \ref{704pm913}). Then, by Lemma \ref{334am513} and \textit{Remark} \ref{704pm913},    (\ref{1012pm87}) becomes
\begin{equation}\label{418pm810}
\sum_{\emptyset \neq S'\subset A}\tilde{c}_{S'}\dashint_c~~\cdots~~ \dashint_c\bigg(\prod_{\substack{\alpha<\beta,\\ \alpha,\beta \in S'}}T_{\beta\alpha}\bigg) I(\bm{\xi}_{S'})f(\bm{\xi})\big|_{\xi_i = 1,\,i \in A\setminus S'}d\bm{\xi}_{S'}d\bm{\xi}_{A^c}
\end{equation}
where
\begin{equation}\label{1225am92}
\tilde{c}_{S'} = q^{n(n-1)/2}\times \frac{q^{\sum_A(A\setminus S') - n|A\setminus S'|}}{p^{{\sum_A(A\setminus S') - |A\setminus S'|(|A\setminus S'| +1)/2}}}.
\end{equation}
Now, we evaluate $f(\bm{\xi})\big|_{\xi_i = 1,\, i \in A\setminus S'}$. Let $<a>$ be the number of elements in $A\setminus S'$ larger than $a\in A^c$. Then,
noting that $pT_{\beta\alpha}\big|_{\xi_{\beta}=1} = 1$ and $S_{\beta\alpha}\big|_{\xi_{\beta}=1} =q/p$,
 we obtain
\begin{equation*}
\begin{aligned}
&f(\bm{\xi})\big|_{\xi_i = 1,\,i \in A\setminus S'} =  p^{(N-n)(N-n-1)/2}\bigg(\prod_{\substack{\alpha<\beta \\\alpha,\beta \in A^c }} T_{\beta\alpha}\bigg) \frac{ \prod_{a \in A^c}\xi_a}{\prod_{a\in A^c}(1-\xi_{a})} W_{t,x,Y_{S' \cup A^c}}(\bm{\xi}_{S' \cup A^c}) \\[15pt]
& \hspace{1cm} \times
 \begin{cases}
 \Big(\displaystyle\prod_{\alpha\in A^c}pT_{N\alpha}\Big) \bigg(\displaystyle \prod_{\substack{\alpha<\beta \neq N\\
\beta \in  S', \,\alpha \in A^c}} S_{\beta\alpha}\bigg)
 \Big(\frac{q}{p}\Big)^{\sum_{a\in A^c}<a>} & ~~\textrm{if}~ N \in S', \\[30pt]
 \bigg(\displaystyle \prod_{\substack{\alpha<\beta \neq N\\
\beta \in  S', \,\alpha \in A^c}} S_{\beta\alpha}\bigg)
 \Big(\frac{q}{p}\Big)^{\sum_{a\in A^c}<a>- |A^c|} & ~~\textrm{if}~ N \notin S'.
\end{cases}
\end{aligned}
\end{equation*}
Hence,  noting that
 \begin{equation}\label{1221am829}
  \Sigma_A(A\setminus S') = \Sigma(A\setminus S') - \sum_{a\in A^c} <a>,
  \end{equation}
 we obtain
\begin{equation}\label{418pm81011}
\begin{aligned}
(\#)~=~&\sum_{\substack{S'\subset A,\\ N \in S'}}q^{n(n-1)/2}p^{(N-n)(N-n+1)/2} \frac{q^{\sum(A\setminus S') - n|A\setminus S'|}}{p^{{\sum(A\setminus S') - |A\setminus S'|(|A\setminus S'| +1)/2}}}\\[5pt]
&\hspace{2cm} \times \dashint_c \cdots \dashint_c\bigg(\prod_{\substack{\alpha<\beta,\\ \alpha,\beta \in S'}}T_{\beta\alpha}\bigg) \bigg(\prod_{\substack{\alpha<\beta,\\ \alpha,\beta \in A^c}}T_{\beta\alpha}\bigg)\bigg(\displaystyle \prod_{\substack{\alpha<\beta \neq N\\
\beta \in  S', \,\alpha \in A^c}} S_{\beta\alpha}\bigg)\bigg(\prod_{\alpha\in A^c}T_{N\alpha}\bigg) \\[10pt]
& \hspace{3cm} \times I(\bm{\xi}_{S'})\frac{ \prod_{a \in A^c}\xi_a}{\prod_{a\in A^c}(1-\xi_{a})} W_{t,x,Y_{S' \cup A^c}}(\bm{\xi}_{S'\cup A^c})d\bm{\xi}_{S'\cup A^c} \\[10pt]
&+\sum_{\substack{ \emptyset \neq S'\subset A,\\ N \notin S'}}q^{n(n-1)/2}p^{(N-n)(N-n-1)/2} \frac{q^{\sum(A\setminus S') - n|A\setminus S'|-|A^c|}}{p^{{\sum(A\setminus S') - |A\setminus S'|(|A\setminus S'| +1)/2-|A^c|}}}\\[5pt]
&\hspace{2cm} \times\dashint_c \cdots \dashint_c\bigg(\prod_{\substack{\alpha<\beta,\\ \alpha,\beta \in S'}}T_{\beta\alpha}\bigg) \bigg(\prod_{\substack{\alpha<\beta,\\ \alpha,\beta \in A^c}}T_{\beta\alpha}\bigg) \bigg(\displaystyle \prod_{\substack{\alpha<\beta \neq N\\
\beta \in  S', \,\alpha \in A^c}} S_{\beta\alpha}\bigg) \\[10pt]
& \hspace{3cm} \times I(\bm{\xi}_{S'})\frac{ \prod_{a \in A^c}\xi_a}{\prod_{a\in A^c}(1-\xi_{a})} W_{t,x,Y_{S' \cup A^c}}(\bm{\xi}_{S'\cup A^c})d\bm{\xi}_{S'\cup A^c}.
\end{aligned}
\end{equation}
\item [(\rmnum{2})] \textbf{Step 2.} In this step, we compute the first sum on the left-hand side of (\ref{331am814}). The first sum   is written
\begin{equation*}
 \sum_{\substack{\textrm{all}\,A\subset\{1,\dots,N\}\\
 \textrm{with}\,|A|=n,\,N \in A}}\bigg(\underbrace{\sum_{\substack{X\,\textrm{with} \\x_{n+1} = x}}~\sum_{\substack{\sigma\,\textrm{with}\\ \sigma(1),\dots, \sigma(n) \in A}} ~ \sum_{\substack{\sigma\,\textrm{with}\\ \sigma(n+1),\dots, \sigma(N) \in A^c,\\\sigma(1),\dots,\sigma(n)\,\textrm{fixed}}} \dashint_c\cdots \dashint_c  [\mathbf{A}_{\sigma}]_{\nu_{n+1},\nu_N}^{-}\Big(\xi_{\sigma(i)}^{x_i - y_{\sigma(i)}-1} e^{(\frac{p}{\xi_i}+q\xi_i - 1)t}\Big)d\bm{\xi}}_{(\#\#)}\bigg).
\end{equation*}
In a similar way to Step 1, we write $(\#\#)$   as
\begin{equation}\label{9099pm88}
\begin{aligned}
&\sum_{z_1,\dots,z_{n}}~\sum_{v_1,\dots,v_{N-n-1}} ~~\sum_{\substack{\sigma\,\textrm{with}\\ \sigma(1),\dots, \sigma(n) \in A}} ~ \sum_{\substack{\sigma\,\textrm{with}\\ \sigma(n+1),\dots, \sigma(N) \in A^c,\\\sigma(1),\dots,\sigma(n)\,\textrm{fixed}}}  \\[8pt]
&\hspace{0.5cm}\dashint_c\cdots \dashint_c [\mathbf{A}_{\sigma}]^-_{\nu_{n+1},\nu_N}\prod_{i=1}^{n}\big(\xi_{\sigma(1)}\cdots\xi_{\sigma(i)}\big)^{-z_i}\prod_{i=1}^{N-n-1} \big(\xi_{\sigma(n+1+i)}\cdots\xi_{\sigma(N)}\big)^{v_i}W_{t,x,Y}(\bm{\xi})d\bm{\xi}.
\end{aligned}
\end{equation}
Note that $[\mathbf{A}_{\sigma}]^-_{\nu_{n+1},\nu_N}=-[\mathbf{A}_{\sigma}]^+_{\nu_{n},\nu_N}$ by Theorem 1.7, Proposition 1.8 in \cite{Lee-Raimbekov-2022} and (\ref{1119pm527}). In (\ref{9099pm88}), we first sum  over $v_1,\dots, v_{N-n-1}$ and then sum over all permutations $\sigma$ with $\sigma(n+1),\dots, \sigma(N) \in A^c$ and $\sigma(1),\dots, \sigma(n)$ fixed. Then, (\ref{9099pm88}) becomes
\begin{equation}\label{10121pm87}
\begin{aligned}
&\sum_{z_1,\dots, z_{n}}\sum_{\substack{\sigma\,\textrm{with}\\ \sigma(1),\dots, \sigma(n) \in A}} \dashint_c\cdots \dashint_c f(\bm{\xi})\bigg(\prod_{\substack{\textrm{all inversions $(\beta,\alpha)$ } \\
\textrm{in}~\sigma(1)\cdots \sigma(n)}} S_{\beta\alpha}\bigg) \prod_{i=1}^{n}\big(\xi_{\sigma(1)}\cdots\xi_{\sigma(i)}\big)^{-z_i}  d\bm{\xi}
\end{aligned}
\end{equation}
where
\begin{equation}\label{418809pm810}
\begin{aligned}
& f(\bm{\xi}) = p^{(N-n)(N-n-1)/2} \Big(\prod_{\alpha\in A^c}pT_{N\alpha}\Big) \bigg(\prod_{\substack{\alpha<\beta \neq N\\
\beta \in A, \,\alpha \in A^c}} S_{\beta\alpha}\bigg) \\
&\hspace{3cm}  \times   \bigg(\prod_{\substack{\alpha<\beta \\\alpha,\beta \in A^c }} T_{\beta\alpha}\bigg) (-1)^{|A^c|+1}I(\bm{\xi}_{A^c}) W_{t,x,Y}(\bm{\xi})
\end{aligned}
\end{equation}
by using the identity (1.6) in \cite{Tracy-Widom-2008}. Here, we note that (\ref{418809pm810})  can be made analytic inside a sufficiently large circle centered at the origin except at the origin for variable $\xi_i$ where $i \in A$ for the use of Lemma \ref{212am827} (see Remark \ref{704pm913}). Then, by Lemma \ref{212am827} and \textit{Remark} \ref{704pm913},   (\ref{10121pm87}) becomes
\begin{equation}\label{4189pm810}
\begin{aligned}
&\dashint_c\cdots \dashint_c f(\bm{\xi})\big|_{\xi_i = 1,\, i \in A}d\bm{\xi}_{A^c}\\
&\hspace{1cm} +\sum_{\emptyset \neq S'\subset A}\tilde{c}_{S'}\dashint_c~~\cdots~~ \dashint_c\bigg(\prod_{\substack{\alpha<\beta,\\ \alpha,\beta \in S'}}T_{\beta\alpha}\bigg) I(\bm{\xi}_{S'})f(\bm{\xi})\big|_{\xi_i = 1,\, i \in A\setminus S'}d\bm{\xi}_{S'}d\bm{\xi}_{A^c}
\end{aligned}
\end{equation}
where $\tilde{c}_{S'}$ is given by (\ref{1225am92}). Now, if we evaluate $f(\bm{\xi})\big|_{\xi_i = 1,\, i \in A}$ and $f(\bm{\xi})\big|_{\xi_i = 1,\, i \in A\setminus S'}$ in (\ref{4189pm810}) by the same method as in Step 1, then we can show
\begin{equation}\label{430pm0901}
\begin{aligned}
(\#\#)~=~ &\sum_{\substack{S'\subset A,\\ N \in S'}}q^{n(n-1)/2}p^{(N-n)(N-n+1)/2} \frac{q^{\sum(A\setminus S') - n|A\setminus S'|}}{p^{{\sum(A\setminus S') - |A\setminus S'|(|A\setminus S'| +1)/2}}}\\[5pt]
&\hspace{1cm} \times \dashint_c \cdots \dashint_c\bigg(\prod_{\substack{\alpha<\beta,\\ \alpha,\beta \in S'}}T_{\beta\alpha}\bigg) \bigg(\prod_{\substack{\alpha<\beta,\\ \alpha,\beta \in A^c}}T_{\beta\alpha}\bigg)\bigg(\displaystyle \prod_{\substack{\alpha<\beta \neq N\\
\beta \in  S', \,\alpha \in A^c}} S_{\beta\alpha}\bigg)\bigg(\prod_{\alpha\in A^c}T_{N\alpha}\bigg) \\[10pt]
& \hspace{2cm} \times J(\bm{\xi}_{S'})I(\bm{\xi}_{A^c})(-1)^{|A^c|}W_{t,x,Y_{S' \cup A^c}}(\bm{\xi}_{S'\cup A^c})d\bm{\xi}_{S'\cup A^c} \\[10pt]
&+~\sum_{\substack{S'\subset A,\\ N \notin S'}}q^{n(n-1)/2}p^{(N-n)(N-n-1)/2} \frac{q^{\sum(A\setminus S') - n|A\setminus S'|-|A^c|}}{p^{{\sum(A\setminus S') - |A\setminus S'|(|A\setminus S'| +1)/2-|A^c|}}}\\[5pt]
&\hspace{1cm} \times\dashint_c \cdots \dashint_c\bigg(\prod_{\substack{\alpha<\beta,\\ \alpha,\beta \in S'}}T_{\beta\alpha}\bigg) \bigg(\prod_{\substack{\alpha<\beta,\\ \alpha,\beta \in A^c}}T_{\beta\alpha}\bigg) \bigg(\displaystyle \prod_{\substack{\alpha<\beta \neq N\\
\beta \in  S', \,\alpha \in A^c}} S_{\beta\alpha}\bigg) \\[10pt]
& \hspace{2cm} \times J(\bm{\xi}_{S'})I(\bm{\xi}_{A^c})(-1)^{|A^c|} W_{t,x,Y_{S' \cup A^c}}(\bm{\xi}_{S'\cup A^c})d\bm{\xi}_{S'\cup A^c}.
\end{aligned}
\end{equation}
Note that $S'$ in the second sum of $(\ref{430pm0901})$ is allowed to be the empty set.
\item [(\rmnum{3})] \textbf{Step 3.} In this step, we sum $(\#)$ and $(\#\#)$. Recall that all $S'$ in the sum $(\#)$ are nonempty sets and  $S'$ in the sum $(\#\#)$ can be the empty set. Using that when $S'\neq \emptyset$,
\begin{equation*}
I(\bm{\xi}_{S'})\frac{ \prod_{a \in A^c}\xi_a}{\prod_{a\in A^c}(1-\xi_{a})} +J(\bm{\xi}_{S'})I(\bm{\xi}_{A^c})(-1)^{|A^c|} = (-1)^{|A^c|} I(\bm{\xi}_{S'\cup A^c}),
\end{equation*}
 we obtain  the sum of $(\#)$ and $(\#\#)$  in the form of
 \begin{equation}\label{4181pm81011}
 \sum_{S' \subset A}c_{A,S'}\dashint_c \cdots \dashint_c K(A^c,S')I(\bm{\xi}_{S'\cup A^c}) W_{t,x,Y_{S' \cup A^c}}(\bm{\xi}_{S'\cup A^c})\bm{d\xi}_{S'\cup A^c}
 \end{equation}
 where
 \begin{equation}\label{412pm818}
{c}_{A,S'} =
 \begin{cases}
 (-1)^{N-n}q^{n(n-1)/2}p^{(N-n)(N-n+1)/2} \frac{q^{\sum(A\setminus S') - n|A\setminus S'|}}{p^{{\sum(A\setminus S') - |A\setminus S'|(|A\setminus S'| +1)/2}}}&~~\textrm{if}~N \in S', \\[15pt]
 (-1)^{N-n}q^{n(n-1)/2}p^{(N-n)(N-n-1)/2} \frac{q^{\sum(A\setminus S') - n|A\setminus S'|-|A^c|}}{p^{{\sum(A\setminus S') - |A\setminus S'|(|A\setminus S'| +1)/2-|A^c|}}}&~~\textrm{if}~N \notin S',
 \end{cases}
 \end{equation}
 and
 \begin{equation*}
  K(A^c,S') =
 \begin{cases}
 \bigg(\prod_{\substack{\alpha<\beta,\\ \alpha,\beta \in S'}}T_{\beta\alpha}\bigg)\bigg(\prod_{\substack{\alpha<\beta,\\ \alpha,\beta \in A^c}}T_{\beta\alpha}\bigg) \bigg(\displaystyle \prod_{\substack{\alpha<\beta \neq N\\
\beta \in  S', \,\alpha \in A^c}} S_{\beta\alpha}\bigg)\bigg(\prod_{\alpha\in A^c}T_{N\alpha}\bigg)&~~\textrm{if}~N \in S', \\[20pt]
  \bigg(\prod_{\substack{\alpha<\beta,\\ \alpha,\beta \in S'}}T_{\beta\alpha}\bigg)\bigg(\prod_{\substack{\alpha<\beta,\\ \alpha,\beta \in A^c}}T_{\beta\alpha}\bigg) \bigg(\displaystyle \prod_{\substack{\alpha<\beta \neq N\\
\beta \in  S', \,\alpha \in A^c}} S_{\beta\alpha}\bigg)&~~\textrm{if}~N \notin S'.
 \end{cases}
 \end{equation*}
\item [(\rmnum{4})] \textbf{Step 4.} Finally, we compute
\begin{equation}\label{437pm818}
\sum_{\substack{A~\textrm{with}\\ |A| = n,\,N \in A}}~\sum_{S' \subset A}{c}_{A,S'}\dashint_c \cdots \dashint_c K(A^c,S')I(\bm{\xi}_{S'\cup A^c}) W_{t,x,Y_{S' \cup A^c}}(\bm{\xi}_{S'\cup A^c})\bm{d\xi}_{S'\cup A^c}.
\end{equation}
Note that if we let $S = S' \cup A^c$, then  $\Sigma(A\setminus S') = \Sigma(S^c)$ and $|A\setminus S'| = |S^c|$. This implies that if a pair $(A,S')$ such that $N \in A \subset \{1,\dots, N\}$ with $|A| = n$ and $S'\subset A$, and another pair $(A',S'')$ such that  $N \in A' \subset \{1,\dots, N\}$ with $|A'| = n$ and $S''\subset A'$ satisfy  $S' \cup A^c = S'' \cup A'^c$, then  $c_{A,S'} = c_{A',S''}$.
Hence, the sum (\ref{437pm818}) is equivalent to
\begin{equation}\label{440pm818}
\begin{aligned}
&\sum_{\substack{S~\textrm{with}\\ |S| > N-n}}~\sum_{(S_1,S_2)}\hat{c}_{S}\dashint_c \cdots \dashint_c K(S_2,S_1)I(\bm{\xi}_{S}) W_{t,x,Y_{S}}(\bm{\xi}_{S})\bm{d\xi}_{S}\\
&\hspace{2cm}+~\sum_{\substack{S~\textrm{with}\\ |S| = N-n,\\N \notin S}}\hat{c}_{S}\dashint_c \cdots \dashint_c K(S,\emptyset)I(\bm{\xi}_{S}) W_{t,x,Y_{S}}(\bm{\xi}_{S})\bm{d\xi}_{S}
\end{aligned}
\end{equation}
where
 \begin{equation*}
 \hat{c}_{S} =
 \begin{cases}
 (-1)^{N-n}q^{n(n-1)/2}p^{(N-n)(N-n+1)/2} \frac{q^{\sum(S^c) - n|S^c|}}{p^{{\sum(S^c) - |S^c|(|S^c| +1)/2}}}&~~\textrm{if}~N \in S, \\[15pt]
 (-1)^{N-n}q^{n(n-1)/2}p^{(N-n)(N-n-1)/2} \frac{q^{\sum(S^c) - n|S^c|-(N-n)}}{p^{{\sum(S^c) - |S^c|(|S^c| +1)/2-(N-n)}}}&~~\textrm{if}~N \notin S
 \end{cases}
 \end{equation*}
 and   $\sum_{(S_1,S_2)}$ is  the sum over all partitions of $S = S_1 \cup S_2$ such that $|S_2|=N-n$ and $N \notin S_2$. By using Lemma \ref{517pm818},
  \begin{equation*}
 \sum_{(S_1,S_2)}K(S_2,S_1) =
 \begin{cases}
 \displaystyle\qbinom{|S|-1}{N-n} \prod_{\substack{\alpha<\beta\\ \alpha,\beta \in S}}T_{\beta\alpha} &~~\textrm{if}~N \in S, \\[30pt]
  \displaystyle\qbinom{|S|}{N-n} \prod_{\substack{\alpha<\beta\\ \alpha,\beta \in S}}T_{\beta\alpha} &~~\textrm{if}~N \notin S
 \end{cases}
 \end{equation*}
 and we note that
 \begin{equation*}
 K(S,\emptyset) = \prod_{\substack{\alpha<\beta,\\ \alpha,\beta \in S}}T_{\beta\alpha}.
 \end{equation*}
 Hence, (\ref{440pm818}) implies  (\ref{331am814}).
\end{itemize}
\end{proof}
(\textit{Continuing the proof of} Theorem \ref{102am830})
Now, finally, we sum (\ref{331am814}) over $n=1,\dots, N$. Note that by the convention  (\ref{510pm822}),
\begin{equation*}
\sum_{n=1}^N\tilde{c}_S =
\begin{cases}
\displaystyle \sum_{n=N - |S| +1}^N\tilde{c}_S&~~\textrm{if}~N \in S,\\[20pt]
\displaystyle \sum_{n=N - |S|}^N\tilde{c}_S&~~\textrm{if}~N \notin S.
\end{cases}
\end{equation*}
If $N \in S$, then
\begin{equation*}
\sum_{n=N - |S| +1}^N\tilde{c}_S = \displaystyle \bigg(\prod_{i=1}^{|S|-1}(q^i-p^i)\bigg)\Big(\frac{q}{p}\Big)^{\Sigma(S^c) - |S^c|(|S^c|+1)/2}
\end{equation*}
by  Lemma \ref{118am92} where $l-1$ is replaced by $|S|-1$ and $k$ is replaced by $N-n$. For the case  of $N \notin S$, observe that
\begin{equation*}
\Big(\frac{q}{p}\Big)^{\Sigma(S^c) - |S^c|(|S^c|+1)/2} = \Big(\frac{q}{p}\Big)^{|S|}\Big(\frac{q}{p}\Big)^{\Sigma(S^c) - |S^c|(|S^c|-1)/2-N}
\end{equation*}
and then apply Lemma \ref{118am92} where $l-1$ is replaced by $|S|$ and $k$ is replaced by $N-n$. Then, we obtain
\begin{equation*}
\sum_{n=N - |S|}^N\tilde{c}_S = \frac{1}{p^{|S|}}\bigg(\prod_{i=1}^{|S|}(q^i-p^i)\bigg)\Big(\frac{q}{p}\Big)^{\Sigma(S^c) - |S^c|(|S^c|-1)/2-N}.
\end{equation*}
Hence,
\begin{equation*}
c_S=
\begin{cases}
\displaystyle \bigg(\prod_{i=1}^{|S|-1}(q^i-p^i)\bigg)\Big(\frac{q}{p}\Big)^{\Sigma(S^c) - |S^c|(|S^c|+1)/2}&~~\textrm{if}~N \in S,\\[20pt]
\displaystyle \frac{1}{p^{|S|}}\bigg(\prod_{i=1}^{|S|}(q^i-p^i)\bigg)\Big(\frac{q}{p}\Big)^{\Sigma(S^c) - |S^c|(|S^c|-1)/2-N}&~~\textrm{if}~N \notin S.
\end{cases}
\end{equation*}
This completes the proof of Theorem \ref{102am830}. \qed

\begin{appendices}
\section{Formulas of $[\mathbf{A}_{\sigma}]_{\nu_n,\nu_N}$ in (\ref{533pm813})}\label{351am129}
In this appendix, we briefly summarize how to find the formulas of $[\mathbf{A}_{\sigma}]_{\nu_n,\nu_N}$ in (\ref{533pm813}), which were developed in  \cite{Lee-2020,Lee-Raimbekov-2022}. For a permutation  $\sigma$ in the symmetric group $\mathcal{S}_N$, $\mathbf{A}_{\sigma}$ is an $N^N \times N^N$ matrix which is obtained as follows. Let $\mathrm{T}_i$ be the simple transposition  which interchanges the $i^{th}$ element and the $(i+1)^{st}$ element and leaves everything else fixed. Then, any permutation $\sigma \in \mathcal{S}_N$ can be written as a product of simple transpositions, that is,
\begin{equation}\label{853pm352023}
\sigma = \mathrm{T}_{i_j}\cdots \mathrm{T}_{i_1}
\end{equation}
for some $i_1,\dots, i_j \in \{1,\dots, N-1\}$. Of course, the form of (\ref{853pm352023}) is not unique. For example, a permutation $(321) \in \mathcal{S}_3$ is written as
\begin{equation*}
(321) =\mathrm{T}_1\mathrm{T}_2\mathrm{T}_1 = \mathrm{T}_2\mathrm{T}_1\mathrm{T}_2.
\end{equation*}
 If $\mathrm{T}_i$ interchanges  $\alpha$ at the $i^{th}$ slot and $\beta$ at the $(i+1)^{st}$ slot in a permutation, that is,
\begin{equation*}
\mathrm{T}_i(~\cdots~ \alpha\beta ~\cdots) = (~\cdots~ \beta\alpha ~\cdots),
\end{equation*}
we denote this $\mathrm{T}_i$  by $\mathrm{T}_i(\beta,\alpha)$ to show explicitly which numbers are interchanged.   For example,
\begin{equation*}
(321) = \mathrm{T}_1(3,2)\mathrm{T}_2(3,1)\mathrm{T}_1(2,1).
\end{equation*}
 Hence, we may have an expression
\begin{equation}\label{123am107}
\sigma = \mathrm{T}_{i_j}(\beta_j,\alpha_j)\cdots \mathrm{T}_{i_1}(\beta_1,\alpha_1)
\end{equation}
for any permutation $\sigma$.
 Given $\mathrm{T}_{i}(\beta,\alpha)$, we define $N^N \times N^N$ matrix $\mathbf{T}_{i}(\beta,\alpha)$ by
\begin{equation*}\label{310am36}
\mathbf{T}_{i}(\beta,\alpha) =\underbrace{\mathbf{I}_N ~\otimes~ \cdots~ \otimes~  \mathbf{I}_N}_{(i-1)~\textrm{times}}~ \otimes~ \mathbf{R}_{\beta\alpha}~ \otimes ~ \underbrace{\mathbf{I}_N ~\otimes~ \cdots ~\otimes~  \mathbf{I}_N}_{(N-i-1)~\textrm{times}}
\end{equation*}
where $\mathbf{I}_N$ is the $N \times N$ identity matrix and  $\mathbf{R}_{\beta\alpha}$ is an $N^2 \times N^2$ matrix, where $\mathbf{R}_{\beta\alpha}$ is defined as follows. Let us label all columns and rows of $N^2 \times N^2$ matrices by $ij$ where $i,j =1,\dots, N$ in the lexicographical order. The $(ij,kl)$ entry of the matrix $\mathbf{R}_{\beta\alpha}$  is defined to be
\begin{equation*}
\big[\mathbf{R}_{\beta\alpha}\big]_{ij,kl} = \begin{cases}
S_{\beta\alpha}& ~\textrm{if}~~ij=kl~\textrm{with}~i=j;\\[3pt]
P_{\beta\alpha}& ~\textrm{if}~~ij=kl~\textrm{with}~i<j;\\[3pt]
Q_{\beta\alpha}& ~\textrm{if}~~ij=kl~\textrm{with}~i>j;\\[3pt]
pT_{\beta\alpha}&~\textrm{if}~~ij=lk~\textrm{with}~i<j;\\[3pt]
qT_{\beta\alpha}&~\textrm{if}~~ij=lk~\textrm{with}~i>j;\\[3pt]
0 &~\textrm{for all other cases}
\end{cases}
\end{equation*}
where
\begin{equation*}
\begin{aligned}
S_{\beta\alpha} =&  -\frac{p+q\xi_{\alpha}\xi_{\beta} - \xi_{\beta}}{p+q\xi_{\alpha}\xi_{\beta} - \xi_{\alpha}},&~P_{\beta\alpha} =  \frac{(p-q\xi_{\alpha})(\xi_{\beta}-1)}{p+q\xi_{\alpha}\xi_{\beta} - \xi_{\alpha}}\\
 T_{\beta\alpha} =& \frac{\xi_{\beta}-\xi_{\alpha}}{p+q\xi_{\alpha}\xi_{\beta} - \xi_{\alpha}},&  Q_{\beta\alpha} =\frac{(p-q\xi_{\beta})(\xi_{\alpha}-1)}{p+q\xi_{\alpha}\xi_{\beta} - \xi_{\alpha}}.
\end{aligned}
\end{equation*}
Now, for a given expression  (\ref{123am107}), define
\begin{equation}\label{932pm35}
\mathbf{A}_{\sigma} := \mathbf{T}_{i_j}(\beta_j,\alpha_j)~\cdots~ \mathbf{T}_{i_1}(\beta_1,\alpha_1).
\end{equation}
It is a known fact that $\mathbf{A}_{\sigma} $ is well-defined in the sense that  (\ref{932pm35}) represents the same matrix for any expression (\ref{123am107}). We have just constructed an $N^N \times N^N$ matrix $\mathbf{A}_{\sigma}$ for a given permutation $\sigma$. Now, let us label all columns and rows of $\mathbf{A}_{\sigma}$  by $\nu=i_1\cdots i_N$ where $i_j \in \{1,\dots, N\},\, j=1,\dots, N$  in the lexicographical order. In this paper, we denote $2\cdots 212\cdots 2$ by $\nu_n$ if $1$ is the $n^{\textrm{th}}$ leftmost number.

Since $\mathbf{A}_{\sigma}$ is a product of matrices, in general, the matrix elements  $[\mathbf{A}_{\sigma}]_{\pi,\nu}$ are possibly written as a sum of the products of the matrix elements of $\mathbf{T}_{i_j}(\beta_j,\alpha_j),\dots, \mathbf{T}_{i_1}(\beta_1,\alpha_1)$ in (\ref{932pm35}). However, it was shown  in \cite{Lee-Raimbekov-2022} that, for some special cases of initial order of particles $\nu$, $[\mathbf{A}_{\sigma}]_{\pi,\nu}$ are expressed as a product of the matrix elements of $\mathbf{T}_{i_j}(\beta_j,\alpha_j),\dots, \mathbf{T}_{i_1}(\beta_1,\alpha_1)$, not a sum of the products. In particular, for the initial order of particles $\nu_N$, that is, $2\cdots 21$, the formulas $[\mathbf{A}_{\sigma}]_{\nu_n,\nu_N}$ are given in Theorem 1.2 (a), Theorem 1.4, Theorem 1.6, Theorem 1.7, and Proposition 1.8 in \cite{Lee-Raimbekov-2022}.
\end{appendices}
\\ \\
\noindent \textbf{Conflict of Interest} The authors declare that they have no conflict of interest.
\\ \\
\noindent \textbf{Funding}
This work was supported by the faculty development competitive research grant (021220FD4251) by Nazarbayev University.
\\ \\
\noindent \textbf{Data Availability}
This work is not based on data, so has no data to share.
\\ \\
\noindent \textbf{Acknowledgment} This paper is based on the master thesis of one of the authors \cite{Zhanibek-2022}.


\begin{thebibliography}{99}
\bibitem{Ayyer-Finn-Roy-2018} {Ayyer, A., Finn, C., Roy, D.:} Matrix product solution of a left-permeable two-species asymmetric exclusion process, Physical Review E, \textbf{97}, 012151, (2018).
\bibitem{Borodin-Wheeler} {Borodin, A. and Wheeler, M.:} Coloured stochastic vertex models and their specctral theory, arXiv:1808.01866.
\bibitem{Borodin-Bufetov} {Borodin, A. and Bufetov A.:} Color-position symmetry in interacting particle systems, Ann. Probab. \textbf{49}(4): 1607--1632, (2021).
\bibitem{Chatterjee-Schutz-2010} Chatterjee, S. and Sch\"{u}tz, G.: Determinant representation for some transition probabilities in the TASEP with second class particles, J. Stat. Phys., {\textbf{140}}, 900--916, (2010).
\bibitem{Derrida-Evans-Hakim-Pasquier-1993} {Derrida, B., Evans, M.R., Hakim, V., Pasquier, V.:} Exact solution of a 1d asymmetric exclusion model using a matrix formulation, Journal of Physics A, \textbf{26}, 1493, (1993).
\bibitem{Evans-Ferrari-Mallick-2009} {Evans, M.R., Ferrari, P.A., Mallick, K.:} Matrix representation of the stationary measure for the multispecies TASEP, Journal of Statistical Physics, \textbf{135}, 217--239, (2009).
\bibitem{Ferrari-Kipnis-1995} {Ferrari, P.A., Kipnis, C.} Second class particles in the rarefaction front, Annales de l'Institut Henri Poincar\'e, \textbf{31}, 143--154, (1995).
\bibitem{Ferrari-Martin} {Ferrari, P.A., Martin, J.B.:} Stationary distributions of multi-type totally asymmetric exclusion processes, The Annals of Probability, \textbf{35}, 807--832, (2007).
\bibitem{Ferrari-Pimentel-2005} {Ferrari,P.A., Pimentel, L.:} Competition interfaces and second class particles, The Annals of Probability, \textbf{33}, 1235--1254, (2005).
\bibitem{Gier-Mead-Wheeler-2021} {de Gier, J., Mead, W. , Wheeler, M.:} Transition probability and total crossing events in the multi-species asymmetric exclusion process, arXive:2109.14232.
\bibitem{Kuan-2020} {Kuan, J.:} Determinantal expressions in multi-species TASEP, Symmetry, Integrability and Geometry: Methods and Applications (SIGMA) \textbf{16}, 133, (2020).
\bibitem{Lee-2010} {Lee, E.:} {Distribution of a particle's position in the ASEP with the alternating initial condition}, J. Stat. Phys., \textbf{140}, 635--647 (2010).
\bibitem{Lee-2017} {Lee, E.:} {Some conditional probabilities in the TASEP with second class particles,} J. Math. Phys., \textbf{58}, 123301, (2017).
\bibitem{Lee-2018} {Lee, E.:} On the TASEP with the second class particles, Symmetry, Integrability and Geometry: Methods and Applications (SIGMA) \textbf{14}, 006, (2018).
\bibitem{Lee-2020} {Lee, E.:} Exact Formulas of the Transition Probabilities of the Multi-Species Asymmetric Simple Exclusion Process, Symmetry, Integrability and Geometry: Methods and Applications (SIGMA), \textbf{16}, 139, (2020).
\bibitem{Lee-Raimbekov-2022} {Lee, E., Raimbekov, T:} {Simplified Forms of the Transition Probabilities of the Two-Species ASEP with Some Initial Orders of Particles}, Symmetry, Integrability and Geometry: Methods and Applications (SIGMA), \textbf{18}, 008, (2022).
\bibitem{Mountford-Giuol-2005} {Mountford, M., Guiol, H.:} The Motion of a Second Class Particle for the Tasep Starting from a Decreasing Shock Profile, The Annals of Applied Probability, \textbf{15}, 1227--1259, (2005).
\bibitem{Nejjar-2020} {Nejjar, N.:} KPZ Statistics of Second Class Particles in ASEP via Mixing, Communications in Mathematical Physics, \textbf{378}, 601--623, (2020).
\bibitem{Prolhac-Evans-Mallick-2009} {Prolhac, S., Evans, M.R., Mallick, K.:} The matrix product solution of the multispecies partially asymmetric exclusion process, Journal of Physics A, \textbf{42}, 165004, 2009.

\bibitem{Tracy-Widom-2008} Tracy, C. and Widom, H.: Integral formulas for the asymmtric simple exclusion process, Communications in Mathematical Physics, \textbf{279}, 815--844, (2008).
\bibitem{Tracy-Widom-2009} Tracy, C. and Widom, H.: On the distribution of a second-class particle in the asymmetric simple exclusion process, Journal of Physics \textbf{A}, \textbf{42}, 425002, (2009).
\bibitem{Tracy-Widom-2013} Tracy, C. and Widom, H.: On the asymmetric simple exclusion process with multiple species, J. Stat. Phys., {\textbf{150}}, 457--470, (2013).
\bibitem{Zhanibek-2022} Tokebayev, Z.: The probability distribution of the species-1 particle in the two-species ASEP with initial configuration $22\cdots21$, MS thesis, Nazarbayev University, (2022).
\end{thebibliography}
\end{document}